\documentclass[letterpaper, 11pt,  reqno]{amsart}
\usepackage[margin=1.2in,marginparwidth=1.5cm, marginparsep=0.5cm]{geometry}

\usepackage{placeins}

\usepackage{amsmath,amssymb,amscd,amsthm,amsxtra, esint}

\usepackage{booktabs} 
\usepackage{microtype}
\usepackage{amssymb}
\usepackage{mathrsfs} 
\usepackage{makecell}

\usepackage{upgreek} 

\usepackage{color}
\usepackage[implicit=true]{hyperref}

\usepackage{multirow}

\usepackage{mathrsfs}

\usepackage{ulem}
\usepackage[makeroom]{cancel}

\usepackage{cases}

\allowdisplaybreaks[2]

\definecolor{gr}{rgb}   {0.,   0.69,   0.23 }
\definecolor{bl}{rgb}   {0.,   0.5,   1. }
\definecolor{mg}{rgb}   {0.85,  0.,    0.85}
\definecolor{yl}{rgb}   {0.8,  0.7,   0.}
\definecolor{or}{rgb}  {0.7,0.2,0.2}

\usepackage{tikz}

\usepackage{array}
\newcolumntype{P}[1]{>{\centering\arraybackslash}p{#1}}
\newcolumntype{M}[1]{>{\centering\arraybackslash}m{#1}}

\usepackage{marginnote}

\usetikzlibrary{shapes.misc}
\usetikzlibrary{shapes.symbols}
\usetikzlibrary{decorations}
\usetikzlibrary{decorations.markings}

\tikzset{
	dot/.style={circle,fill=black,draw=black,inner sep=0pt,minimum size=0.5mm},
	>=stealth,
	}

\tikzset{
	dot2/.style={circle,fill=black,draw=black,inner sep=0pt,minimum size=0.2mm},
	>=stealth,
	}

\tikzset{
	ddot/.style={circle,fill=white,draw=black,inner sep=0pt,minimum size=0.8mm},
	>=stealth,
	}

\tikzset{decision/.style={ 
        draw,
        diamond,
        aspect=1.5
    }}

\tikzset{dia2/.style
={diamond,fill=white,draw=black,inner sep=0pt,minimum size=1mm},
	>=stealth,
	}

\tikzset{dia/.style
={star,fill=black,draw=black,inner sep=0pt,minimum size=1mm},
	>=stealth,
	}

\tikzset{dia/.style
={diamond,fill=black,draw=black,inner sep=0pt,minimum size=1.3mm},
	>=stealth,
	}

\makeatletter

\def\<#1>{\xusebox{#1}}
\makeatother

\tikzset{
    position/.style args={#1:#2 from #3}{
        at=(#3.#1), anchor=#1+180, shift=(#1:#2)
    }
}

\newtheorem{theorem}{Theorem} [section]

\newtheorem{lemma}[theorem]{Lemma}
\newtheorem{proposition}[theorem]{Proposition}
\newtheorem{remark}[theorem]{Remark}

\newtheorem{definition}{Definition}
\newtheorem{corollary}[theorem]{Corollary}

\newtheorem*{ack}{Acknowledgments}

\DeclareMathOperator*{\supp}{supp}

\newcommand{\om}{\omega}
\newcommand{\Om}{\Omega}
\newcommand{\noi}{\noindent}
\newcommand{\Z}{\mathbb{Z}}
\newcommand{\R}{\mathbb{R}}

\newcommand{\T}{\mathbb{T}}

\newcommand{\N}{\mathbb{N}}

\newcommand{\F}{\mathcal{F}}

\newcommand{\dl}{\delta}

\newcommand{\eps}{\varepsilon}
\newcommand{\kk}{\kappa}
\newcommand{\g}{\gamma}

\newcommand{\ft}{\widehat}

\newcommand{\wt}{\widetilde}

\newcommand{\dx}{\partial_x}
\newcommand{\dt}{\partial_t}

\renewcommand{\H}{\mathcal{H}}
\newcommand{\sgn}{\textup{sgn}}
\newcommand{\cK}{\mathcal{Q}}

\newcommand{\Od}{\Om^{({\rm d}, \dl)}}

\newcommand{\Os}{\Om^{({\rm s}, \dl)}}

\newcommand{\Sdt}{S_{\dl}^{{\rm (d)}}(t) }

\newcommand{\Sst}{S_{\dl}^{{\rm (s)}}(t) }

\renewcommand{\l}{\ell}

\newcommand{\les}{\lesssim}
\newcommand{\ges}{\gtrsim}
\newcommand{\ZP}{\Z_{\geq 0}}

\newcommand{\jb}[1]
{\langle #1 \rangle}

\numberwithin{equation}{section}
\numberwithin{theorem}{section}

\newcommand{\Rst}{ \R\backslash\{0\}  }

\newcommand{\G}{\mathcal{G}_\delta}

\newcommand{\M}{\mathcal{M}}

\newcommand{\Mssss}{M^{\frac 34,\dl}_{T}}

\newcommand{\too}{\longrightarrow}

\newcommand{\1}{{\mathbf 1}}

\newcommand{\pshallow}{ p^{\rm (s)} }
\newcommand{\pdeep}{ p^{\rm (d)} }

 \DeclareMathAlphabet{\mathpzc}{OT1}{pzc}{m}{it}

\makeatletter
\@namedef{subjclassname@2020}{%
 \textup{2020} Mathematics Subject Classification}
\makeatother

\begin{document}

\baselineskip = 14pt

\title[Convergence of the ILW equation]
{Deep-water and shallow-water limits of the intermediate long wave equation}

\author[G.~Li]
{Guopeng Li}

\address{
Guopeng Li\\
School of Mathematics and Maxwell Institute for Mathematical Sciences\\
The University of Edinburgh\\
James Clerk Maxwell Building\\
The King's Buildings\\
Peter Guthrie Tait Road\\
Edinburgh\\ 
EH9 3FD\\United Kingdom} 
\address{
Department of Mathematics and Statistics\\
 Beijing Institute of Technology\\
 Beijing\\ China} 
\email{guopeng.li@ed.ac.uk}

\subjclass[2020]{35A01, 35A02, 35Q53, 35Q35.}

\vspace*{1mm}

\keywords{Intermediate long wave equation,
Benjamin-Ono, 
Korteweg-de Vries
equation,
limiting behaviours,
energy method, well-posedness}

\begin{abstract}

In this paper, we study the low regularity convergence problem for the intermediate long wave equation (ILW), with respect to the depth parameter $\delta>0$, on the real line and the circle. As a natural bridge between the Korteweg-de Vries (KdV) and the Benjamin-Ono (BO) equations, the ILW equation is of physical interest. We prove that the solutions of ILW converge in the $H^s$-Sobolev space for $s>\frac12$, to those of BO in the deep-water limit (as $\delta\to\infty$), and  to those of KdV in the shallow-water limit (as $\delta\to 0$). This improves previous convergence results by Abdelouhab, Bona, Felland, and Saut (1989), which required $s>\frac32$ in the deep-water limit and $s\geq2$ in the shallow-water limit. Moreover, the convergence results also apply to the generalised ILW equation, i.e.~with nonlinearity $\partial_x (u^k)$ for $k\geq 2$. Furthermore, this work gives the first convergence results of generalised ILW solutions on the circle with regularity $s\geq \frac34$. Overall, this study provides mathematical insights for the behaviour of the ILW equation and its solutions in different water depths, and has implications for predicting and modelling wave behaviour in various environments.

\end{abstract}

\maketitle
%
%
%
%

\section{Introduction}
\label{SEC:1}

\subsection{Background}
\label{SUB:bac}

The rigorous theory of internal wave propagation at the interface between two layers of immiscible fluids of differing densities has garnered significant attention from both mathematical and physical studies.
This is due to the system's simplicity as an idealisation of internal wave propagation, its challenging nature from a modelling perspective, and the mathematical and numerical difficulties that arise when analysing the system.
The Korteweg-de Vries equation (KdV) and the Benjamin-Ono equation (BO) are the most fundamental models for shallow-water wave and deep-water wave propagations, respectively. The intermediate long wave equation (ILW), on the other hand, models wave  behavior in different water depths, which builds a model-theoretical bridge between the BO equation and the KdV equation.
 The recent book by Klein-Saut  \cite{KS2022} and the survey article by Saut  \cite{Saut2019} provide a comprehensive overview of the subject, along with relevant references.

To be more precise, the ILW equation is a one-way propagation asymptotic model that describes internal waves at the interface between two layers of immiscible fluids, under the rigid lid assumption and with a flat bottom. The depth parameter is defined by the relative depths of the two fluid layers, and the interface between the layers is approximately governed by the ILW equation. It is therefore natural to expect that as the depth tends to zero and infinity, the ILW should converge to KdV (shallow-water limit) and BO (deep-water limit) respectively. However, the rigours justification of such convergences, in particular in the low regularity regime, raises mathematical challenges, which will be the main concern of this paper.

The convergence problem of ILW is rooted in the study of water waves \cite{Ben67, Benn66} and has drawn huge attention in recent years due to its wide connection with other branches of science, such as internal gravity waves, oceanography, atmospheric science and quantum field theory
\cite{KB, SAF, SH, Lip, LHA, CMH, MR80, MT88, OB, Ro, BLL1, BLL2}.
The derivation depends variously on the wave amplitude, wavelengths, and depth ratio of the two layers,
see \cite{BLS, CGK}.
In particular, each of the finite-depth solitary waveforms, wave speeds, and wavelengths varies with the depth parameter continuously, bridging the two limiting situations \cite{KKD}.
As a result,  
ILW provides a good understanding of the wave motions in different water depths and can be useful for various practical applications, such as predicting the behavior of ocean waves and designing coastal defense structures.

 When studying the nonlinear dispersive equations, it is important to comprehend the interplay between the nonlinear and dispersive effects that determine the behavior of solutions. This understanding is crucial in order to fully grasp the dynamics of the solutions. The ILW model, in particular, has been compared to laboratory experiments, as demonstrated in previous studies such as \cite{KKD, BLS}.
 Back to our convergence issue,
we are also interested in the limiting behaviour of fully nonlinear models of the ILW-type.
The fully nonlinear evolution models derived by Matsuno \cite{MAT92} highlights their importance in the modelling aspect.  
Thus, the study of the convergence of full nonlinear models has the potential to advance our knowledge in this field and contribute to a better understanding of internal wave propagation in general.
Other interesting convergence features of
the ILW model, such as 
the $N$-soliton solutions, Hamiltonian structure,
recursion scheme for the infinite number of conservation laws,
and an inverse scattering problem, etc; see \cite{AT, CL, JE, KUP, LR, CGK, SAK, KSA}.

%
%
%
%
%

From the mathematical perspective,
the convergence of the ILW solitary wave solutions has been well understood in the 1970s and 1980s \cite{AS, Jos, KKD, AFSS, Milo}.
Moreover, the   numerical simulations of ILW convergence behaviours in \cite{KKD} and 
 the validity of deep-water limit in \cite{BD00},
suggested that the convergence of the ILW dynamics should hold 
not only for  the solitary wave solutions but also
for a general class of solutions.
Later in \cite{ABFS}, 
the convergence of ILW solutions were verified in $H^s$-Sobolev spaces with sufficient high regularity (see more discussion below).
In this work, our aim is to establish a suitable approach to study the low regularity ILW convergence problems and  
also the method is  capable of handling  ILW-type  associated with  general nonlinearities (see Section~\ref{SEC:limit1}).
The results in this paper represent  the {\it first low regularity convergence} for ILW-type dynamics on the torus, 
however, there is still wide range open until we reach the critical space  $H^{-1/2}$ 
(on both $\R$ and $\T$, with any depth parameter $\dl>0$),
which is recently identified by the author and his collaborators \cite{CFLOP24}.
They showed ill-posedness in $H^s$ when $s < -1/2$
in the sense of failure of
continuity of the data-to-solution map, and proved a-priori bounds on smooth solutions for
$-1/2<s<0$.
It is also worth mentioning here that the author and his collaborators, 
Oh, Zheng, and Chapouto, have contributed to the field by studying 
the convergence of ILW-type dynamics from a statistical perspective \cite{LOZ, CLOZ}. 

%

%
%
%
%
In the past decades, there has been significant progress in the study of BO and KdV with low regularity data, 
see for instance \cite{KPV96, MP12, BIT11, ZH97, IK07}. In particular, both BO and KdV are globally well-posed in $L^2$. However, to our best knowledge, the rigorous mathematical justification of the convergence of the ILW (or ILW-type equation), in particular in the low regularity regime, is still widely open. The main purpose of this paper is to improve our understanding along this line of research.
The first technique used to justify the ILW dynamical convergence in \cite{ABFS} was based on the classical energy method. However, this approach did not 
make use of the dispersive effects
and as a result, a regularity restriction of $s>\frac32$ was required to construct uniform control over the deep-water solutions, while a higher regularity restriction of $s\geq 2$ (via higher-order conservation laws) was needed to construct uniform control over the shallow-water solutions.
Bourgain's Fourier restriction norm method \cite{B93a, B93b} enables us to study the low regularity initial data problems. However, this method is not suitable for our convergence problem, as the solution space $X^{s,b}$-type (as defined in \eqref{deX}) depends on the depth parameter and therefore not suitable for comparing different solutions with different fluid depths.
The concept of ``unconditional well-posedness" introduced by Kato \cite{KA95, KA96} allows for the construction of solutions in $C_TH^s$
regardless of the dispersive for fixed depth parameters, but still not enough for the convergence problem.
Nevertheless, even the combination of the Fourier restriction norm method and unconditional well-posedness is not sufficient for our convergence problem. The main novelty of the argument presented in Section~\ref{SEC:limit1} is that we must always ensure that the difference between two solutions (corresponding to different depth parameters) can be absorbed by leveraging the structure of the equation and our choice of function space. In particular, further development of the ILW dispersion structure is required, as stated in Lemma~\ref{LEM:ca1b} and Lemma~\ref{LEM:ca1c}. Finally, by combining all of these ideas, we construct a perturbative analysis to establish our desired convergence of the ILW-type of dynamics.

This work is important for improving our understanding of the behaviour of internal wave propagation at the interface between two immiscible fluids of differing densities and has practical implications for predicting and modelling wave behaviour in various water depths.
Moreover,
by providing a rigorous mathematical justification for the convergence of the ILW-type equation  with rough initial data, this paper aims to contribute to the field and advance our understanding of internal wave propagation. Additionally, this study aims to bridge the gap between the mathematical and physical communities and has the potential to inspire future research and practical applications.

\subsection{Intermediate long wave equation}
The ILW equation is given by:
\begin{align}
\label{ILW1}
\begin{cases}
\dt u -
 \G \partial_x^2 u  =  \dx (u^2)    \\ 
u|_{t = 0} = u_0
\end{cases}
\qquad ( t, x) \in \R \times \M,
\end{align}

\noi
where $0<\delta<\infty$,  $u:\R\times \M\to \R$
and $\M =\R$ or $
\T = \R/(2\pi \Z)$.
Here, the operator $\G$ is defined by
\[
\G=-\coth(\dl \dx)-\dl^{-1} \dx^{-1},
\]
which characterises the phase speed and it
 is understood as the Fourier multiplier  by
\begin{equation*}
\label{OP}
\ft{\G f}(n)
:=-i \Big( \coth( \delta n )-\frac{1}{ \delta
n}    \Big)
 \ft f(n)
 \qquad \text{for $ n \in \widehat{\M} $} ,
\end{equation*}

\noi
$\coth(x)$ is the hyperbolic cotangent function defined by 
$
\coth (x)=\frac{e^{x}+e^{-x}}{e^{x}-e^{-x}}, $ 
$x\in \R\backslash\{0\},$ with the convention $\coth(x)-\frac 1x=0$ for $x=0$,
and
$\widehat \M$ is the 
Pontryagin dual of $\M$, i.e., $\widehat \M=\R$, when $\M=\R$, and 
$\widehat \M=\Z$, when $\M=\T.$

\begin{remark}
	\rm
	
%
%
%
	
%
%
	 Joseph \cite{Jos} showed that the ILW equation is a special form of the Whitham equation (on $\mathbb{R}$) \cite{W67} 
	\begin{equation*}
\dt u + \dx \int_{-\infty}^{\infty}  K(x-y)  u(t,y) dy = \dx(u^2),
\end{equation*}	 
by utilising the dispersion relation derived in \cite{Phis} and it can be seen by considering the ILW operator $\G$ as an integral kernel (on $\mathbb{R}$):
\begin{align*}
	\G f(x)&=
	-\frac{1}{2\dl} \,\text{\,p.v.} \int_{-\infty}^\infty  \Big[ \coth\Big(\frac{\pi(x-y)}{2\dl} \Big) - {\rm sgn}(x-y)  \Big] f(y) dy.
	\end{align*}

%
%
%
%
%
%

\end{remark}

\subsection{Deep-water and shallow-water limits of generalised ILW}
\label{SUB:LE}

  In the following,
we consider the generalised intermediate long wave equation (gILW) on $\M$:
\begin{equation}
\label{gILW1}
\begin{cases}
\dt u -   \G  (\dx^2 u)= \dx(u^k) \\
u|_{t = 0} = u_0,
\end{cases}
\qquad (t, x) \in \R \times \M,
\end{equation}

\noi
where $k \geq  2$ is an integer. 
When $k = 2$, the equation \eqref{gILW1} corresponds to ILW \eqref{ILW1}, while,
when $ k = 3$, it is known as the modified ILW equation.


Our main goal is to study the deep-water limit $(\dl \to \infty)$ and the shallow-water limit $(\dl \to 0)$
of solutions to gILW \eqref{gILW1} with rough initial data.
 In the following, let us briefly go over the formal derivation of the
limiting equation in each of the deep-water and shallow-water limits,  for further details, we refer readers   to \cite{LOZ}.
With a slight abuse of notation,
\begin{align*}
\ft \G(n) 
=    - i   \Big( \coth(\dl n)-\frac{1}{\dl n} \Big).
\end{align*}

\noi
\textbf{Deep-water limit} 
($\delta \to \infty$)

\smallskip
\indent
In this case, one can show that
\begin{align*}
\lim_{\dl\to\infty}\ft \G(n)
= - i \sgn(n) 
\end{align*}

\noi
for any $n \in \M$.
The deep-water limit is  sending $\dl \to \infty$, 
and
the  gILW equation \eqref{gILW1}
converges to 
the following generalised BO (gBO) on~$\M$:
\begin{align}
\partial_t u     -  \mathcal{H} (\partial_x^2 u)
=  \dx( u^k )   ,
\label{gBO}
\end{align}

\noi
where 
$\mathcal{H}$ is the spatial Hilbert transform 
defined
by $
\ft{\H f}(n)=-i \sgn(n)\ft f(n) .$
Formally speaking, 
one can view the gILW equation \eqref{gILW1} as
the perturbed gBO equation 
\begin{equation}
\label{PE1}
\partial_t u - \H  (\partial_x^2 u) +
\cK_\dl \dx u
=  \dx(u^k),
\end{equation}

\noi
where  $\cK_\dl = (\H-  \G )\dx $ is defined 
as a Fourier multiplier operator with symbol
\begin{align}
q_{\dl}(\xi) 
=   \dl^{-1} -  \xi  \coth(\dl \xi)+ |\xi|.
\label{PO1a}
\end{align}
In order to prove rigorous convergence, it is necessary to show that $\cK_\dl \dx$ tends to zero in some suitable sense.
 In view of equation~\eqref{PO1a}, we have $|q_{\dl}(\xi)|\leq \frac1\dl$, which suggests that in the deep-water regime $\dl \gg1$, long waves with relatively small frequencies $|n| \ll \dl$ closely approximate long waves in infinitely deep water ($\dl = \infty$).

\medskip
\noi
\textbf{Shallow-water limit} ($\dl\to 0$).
\smallskip

\indent
By using the power series of $\coth(x)$,
a direct computation shows that, 
for $n \in\Rst  $, we have 
\begin{align}
\begin{aligned}
\ft{\G \dx^2 u}(n) 
& =  i \Big( \coth( \delta n )-\frac{1}{ \delta n} \Big)
n^2 \ft{u}(n) \\
&  =  i \frac{ \dl}{3} n^3 \ft{u}(n) + o(1), 
\end{aligned} \label{asy1}
\end{align}

\noi
as $\dl \to 0$.
The identity \eqref{asy1} shows that, 
the dispersion in \eqref{ILW1} disappears as $\dl \to 0$, 
formally yielding the inviscid Burgers equation\footnote{inviscid Burgers’ equation: $
	\dt u +  \dx(u^2) = 0$.
}
in the limit.
In order to circumvent this issue, we introduce the following scaling transformation
for each $\dl > 0$, \cite{ABFS}:
\begin{align}
v(t,x) = {3}{ \dl^{\frac{1}{1-k}} } u( {3}{\dl^{-1} } t,x),
\label{trans}
\end{align}

\noi
which leads to the following scaled gILW:
\begin{equation}
\label{ILW3}
\dt v   -  \frac{3}{ \dl}  \G  \dx^2 v= \dx(v^k) .
\end{equation}

\noi
Namely, $v$ is a solution to the scaled gILW \eqref{ILW3}
(with the scaled initial data)
if and only if $u$ is a solution to the original gILW \eqref{ILW1}.
In view of~\eqref{asy1},
the scaled gILW~\eqref{ILW3} formally converges to the  following generalised
KdV equation (gKdV) on~$\M$:
\begin{align}
\dt v + \dx^3 v = \dx(v^k) . 
\label{KDV}
\end{align}
We remark here that it is natural and physically meaningful   to perform
the scaling transformation \eqref{trans}.
See discussions in \cite[p.\,5]{LOZ}, \cite[(1.7)]{SC05} and \cite{SAK}.

%
%
%
%
%
%
%
%
%
%
%
%
%
%
%
%
%
%
%
%
%

\subsection{Main results}

In the work of Abdelouhab, Bona, Felland, and Saut \cite{ABFS}, it was shown that the (scaled) ILW dynamics converges to the BO dynamics in the deep-water limit and to the KdV dynamics in the shallow-water limit. However, these results were limited to high-regularity solutions, with convergence established in $C(\mathbb{R}; H^s(\M))$ for $s>\frac32$ (as $\dl\to\infty$) and $s\geq 2$ (as $\dl\to 0$), respectively. 
The objectives of this paper are twofold:
(i) to extend the convergence results to low-regularity solutions with $s>\frac12$ for the ILW dynamics, in which the nonlinearity is $\dx(u^2)$,
and (ii) to incorporate the convergence to the gILW dynamics with $s>\frac34$, where nonlinearity is $\dx(u^k)$ for $k\geq2$,
 in both deep-water and shallow-water limits (as stated in Theorems \ref{THM:1} and \ref{THM:2}). 
In particular, this establish  the first convergence results with rough periodic data. 
For additional information regarding the convergence in $\mathbb{R}$, please refer to the works \cite{GW, HW}.

The approach for establishing the convergence of   ILW-type consists of two steps. For clarity in the explanation, we will focus our discussion on the deep-water limit of the gILW equation \eqref{gILW1} (unless otherwise specified).

\smallskip
\noi
{\bf Step 1:}
{\bf Establish the uniform in $\dl$ control over solutions.}\rule[-2.5mm]{0mm}{0mm}
\\
\indent
To construct a solution $u_{\dl}$ for the   gILW equation \eqref{gILW1} for a given initial data $u_{0}$ and a fixed parameter $\dl > 0$, we employ the method developed by Molinet-Tanaka \cite{MT}.
This directly implies the local well-posedness of the gILW equation \eqref{gILW1} for a fixed depth parameter $0\leq \dl \leq \infty$, where $\dl=\infty$ corresponds to the gBO equation and $\dl=0$ corresponds to the gKdV equation. In particular, the following lemma holds for a fixed depth parameter:
\begin{lemma}[Fixed $\dl$ well-posedness \cite{MT}]
	\label{LM:LWP}
	Let $s\geq \frac34$ and $k\geq 2$. Then, for any fixed  $0\leq\dl\leq \infty$,
	the gILW equation \eqref{gILW1} is   unconditionally locally well-posed in $H^{s}(\M)$.
	The maximal time of existence $T=T( \|u_{0}\|_{H^{\frac34}(\M)},\dl) > 0$ depends on the initial data and the parameter $\dl$.
\end{lemma}

To take the limit as the parameter $\dl$ approaches infinity, we prove that the solution map of the gILW equation \eqref{gILW1} is independent of $\dl$. Specifically, the local existence time $T$ does not depend on $\dl$, as stated in Theorem \ref{THM:1}. 
This can be achieved by upgrading Lemma \ref{LM:LWP} to be uniformly in $\dl$ and as a direct consequence of the uniform 
  local well-posedness, we can extract an uniform control over solutions of  gILW \eqref{gILW1}.

\smallskip
\noi
{\bf Step 2:}
{\bf Convergence of the gILW dynamics at the single trajectories.}\rule[-2.5mm]{0mm}{0mm}
\\
\indent
To show the convergence of the gILW solution,
we develop a perturbative argument in Section \ref{SEC:limit1}, which heavily relies on the structure of the ILW-type equation and the uniform (in $\dl$) bounds over the solution. 
Our goal is to prove that the family of gILW solutions $\{u_{\dl} \}_{\dl\geq1}$ forms a Cauchy sequence\footnote{
Alternatively, we can also directly take the difference between the gILW solution and the gBO solution and show it converges to 0 in an appropriate manner.
	Since we already know that gBO is the limiting equation.}
in $C([0,T]; H^{s}(\T))$.
Firstly, as stated in Lemma \ref{LEM:p2b}, the linear dispersion of the ILW-type equation behaves like the BO-type, uniformly for any $2\leq \dl \leq \infty$. 
we can reformulate the gILW equation as a perturbed gBO equation and then take the difference between two perturbed equations as presented in \eqref{ILW5}. 
Moreover, by a standard argument as in constructing the energy estimate, 
our analysis reduces to estimating the linear perturbation and nonlinear interaction. 
The linear perturbation is controlled by further exploring the structure of the ILW-type dispersion, 
while the nonlinear interaction appears as an energy-type estimate. 
Thus, we establish the convergence of our gILW solutions in the deep-water limit.

Theorems \ref{THM:1} and \ref{THM:2} establish
the first convergence result for the ILW equation \eqref{ILW1} with low regularity, representing an improvement over the previous work in \cite{ABFS}, which required $s>\frac32$, to the current requirement of $s>\frac12$. Additionally, we have established the convergence result for the gILW equation \eqref{gILW1} with a regularity of $s\geq \frac34$. This represents the first result of its kind on the torus $\T$.

Let $0<T<1$, we denote $\Phi^{\rm(d)}_{T, \dl}$ to  be the flow map for the gILW equation \eqref{gILW1}, which was constructed in \cite{MT} for fixed $\dl$.
For every subset $A\subset H^s$,
we define the flow map as follows:
\begin{equation}
\label{eq:Phi}
\Phi^{\rm(d)}_{T,\dl}(A)=\{  u(t,.)   \in H^s   |  \hbox{ where }
u(t,.) \hbox{ solves } \eqref{gILW1} \hbox{ for } 0<t\leq T \hbox{ with } u(0,.)\in A\}.
\end{equation}
With a slight change on the subscript of \eqref{eq:Phi} we 
denote $ \Phi_{T, {\infty}} $ to be the flow map for the gBO equation \eqref{gBO}.
The first contribution of this paper is the deep-water convergence:

\begin{theorem}[Deep-water theory]
	\label{THM:1}
	Let $k\geq 2$ and $u_{0}\in H^{s}(\M)$ for
	$s\geq \frac{3}{4}$, where $\M=\R \textup{ or }\T$.
	 Then, the following statements hold.
	 
				\smallskip
	\begin{itemize}
		\item[(i)]	
Let $2\leq \dl\leq \infty$. Then, for any $0<T<1$
 the solution map $\Phi^{\rm (d)}_{T,\dl}$ satisfies
\begin{align*}
	\| \Phi^{\rm (d)}_{T,\dl} (u_{0}) \|_{C([0,T];H^{s}(\M))} \leq 
	C(\|u_{0}\|_{H^{s}(\M)} ).
	\end{align*}

	\noi	
	The
	solution map $\Phi^{\rm (d)}_{T,\dl}: u_{0} \to u_{\dl} $ 
	is continuous from $H^{s}(\M)$ to $C([0, T ]; H^{ s}(\M) )$, uniformly on $\dl \in [2, \infty]$.
	Moreover,
	the local existence time 
	$T=T( \|u_0\|_{H^{\frac34}(\M)})>0$ is
	independent of $\dl$.

\smallskip
\item[(ii)]
Let
$ \Phi^{\rm (d)}_{T,\dl}   (u_{0})  = u_{\dl} $  denotes the solution of  gILW \eqref{gILW1} and
 $	\Phi_{T,\infty} ( u_{0})  = u_{\infty}$ denotes the solution of 
  gBO \eqref{gBO}.
	Then, we have 
	\begin{align*}
	\lim_{\dl \to \infty}
	\| u_\dl- u_{\infty} \|_{C([0,T]; H^s(\M))}=0.
	\end{align*}

	\end{itemize}

When we only consider $k = 2$, 
the statements {\rm (i)} and {\rm (ii)} hold true for $s> \frac{1}{2}$.
	
\end{theorem}

%
%
It is noteworthy that 
the regularity $s\geq \frac34$ in
Theorem \ref{THM:1} (and see Theorem \ref{THM:2} below)
 is needed to deal with general nonlinearity $\dx(u^k)$, which
encompass the convergence of the ILW solutions.
However, when considering only the ILW equation with a quadratic nonlinearity $\dx(u^2)$,
we observe improved low regularity convergence results for $s>\frac12$.
These regularity restriction comes from the step of establishing uniform control over the solutions,
which we utilised the method that was introduced in the works \cite{MV15, MT}.

For each fixed value of $2 \leq \dl \leq \infty$, the construction of the gILW solution  \eqref{gILW1}
we saw in Lemma \ref{LM:LWP}.
To show the uniform control of the  solutions  with respect to $\dl$ for any $2\leq \dl\leq \infty$,
it is necessary to observe the following dispersion structure of the ILW-type equation:
\begin{align*}
p_{\dl}^{\rm (d)}(n) \sim |n|^{2}
\qquad
\text{ for }
n\in \Z,
\end{align*}

\noi
where $p_{\dl}^{\rm (d)}(n)$ is defined in \eqref{dispersionP}.
By revisiting the argument presented in \cite{MV15, MT}, we can verify that all relevant estimates are uniformly in $\dl$ for any $2\leq \dl\leq \infty$. 
Thus, the uniform control of the solutions is obtained as a direct outcome of the uniform local well-posedness.

Let  $u_{\dl}$ denotes 
the solution of  \eqref{gILW1}, to establish deep-water convergence, we will first prove that the sequence of solutions 
 $\{u_{\dl}\}_{\dl\geq1}$ is Cauchy in a weaker function space $C_TH^{s-1}$, as $\dl\to \infty$.
 And then, by a standard truncation argument we upgrade it to our desired space.
 This is achieved by using a perturbative argument starting with the difference of two equations with respect to the different fluid depth.
 Then, we separate issues into linear perturbation and nonlinear perturbation.
When we estimate the nonlinear interaction of the difference equation, 
it suffices to the energy-type estimates of difference solutions $u_{\g}$ and $u_\dl$, for $\dl\not =\g$.
The primary difficulty is that  one  needs to place solution $u_{\g}$ in the function space  $M^{s,\dl}_{T}$, where the depth parameters are not matching and it is generally unbounded.
To address this issue, the structure of the ILW-type equation is further utilised in the deep-water regime. This allows for the uniform control of any perturbations in the dispersion for $\dl \geq 2$. Specifically, consider $u_{\g}\in M^{s,\g}_{T}$ as a solution to $g\textrm{ILW}{\g}$. Then, we have
$
\| u_{\g} \|_{M_{T}^{s,\dl}}
\les \|u_{\g} \|_{M^{s,\g}_{T}},
$
see Lemma \ref{LEM:ca1b} for details. 
We point  out that
the perturbative analysis developed in
Section \ref{SEC:limit1}
is generally applicable to the both deep and shallow water cases.
However, new difficulties arise as we analyse the linear perturbation due to the singular behavior as $\dl \to 0$. We will address these challenges in the upcoming discussion.

By following  the definition of \eqref{eq:Phi}, we define  the solution maps for scaled gILW \eqref{ILW3}
to be $\Phi^{\rm (s)}_{T,\dl}$ 
and gKdV  \eqref{KDV}  to be $\Phi_{T, {0}}$.
The second contribution of this paper is the shallow-water convergence:

\begin{theorem}[Shallow-water theory]
	\label{THM:2}
	Let 
$k\geq 2$ and 	$v_{0}\in H^{s}(\M)$ for
	$s\geq \frac{3}{4}$, where $\M=\R \textup{ or }\T$.
	 Then, the following statements hold.

	\smallskip
	\begin{itemize}
		\item[(i)] 
Let $0 <\dl<1$. Then, for any $0<T<1$
 the solution map  $ \Phi^{\rm (s)}_{T,\dl} $ satisfies 
	\begin{align*}
	\| \Phi^{\rm (s)}_{T,\dl} (v_{0}) \|_{C([0,T];H^{s}(\M))} \leq 
	C( \|v_{0}\|_{H^{s}(\M)} ).
	\end{align*}

	\noi
	The
	solution map $\Phi^{\rm (s)}_{T,\dl}: v_{0} \to v_{\dl} $ 
	is continuous from $H^{s}(\M)$ to $C([0, T ]; H^{ s}(\M) )$, uniformly on $\dl \in (0, 1)$.
	Moreover,
	the local existence time 
	$T=T( \|v_0\|_{H^{\frac34}(\M)})>0$ is
	independent of $\dl$.

\smallskip
		
		\item[(ii)]
		Let   $\Phi^{\rm (s)}_{T,\dl} (v_{0})  = v_{\dl}  $ denotes the solution of
  scaled gILW \eqref{ILW3}   and $\Phi_{T,{0}} (v_{0})  = v_{\rm gKdV}    $
	denotes	the solution of gKdV \eqref{KDV}.
	Then, we have 
	\begin{align*}
	\lim_{\dl \to 0}
	\| v_\dl- v_{\rm gKdV} \|_{C([0,T]; H^s(\M))}=0.
	\end{align*} 
	\end{itemize}

When we only consider $k = 2$, 
the statements {\rm (i)} and {\rm (ii)} hold true for $s> \frac{1}{2}$.
	
\end{theorem}

In the shallow-water regime,
 to show the uniform control of the solution, it is important to note that while the limiting equation (gKdV) is locally well-posed in $H^{s}$ for $s\geq \frac23$ (as established in \cite{MT}), the scaled gILW equation is locally well-posed in $H^{s}$ for $s\geq \frac34$ for a fixed value of $\dl>0$.
In particular, for the scaled gILW, we have  the following relationship:
\begin{align}
\label{eq:slow0}
\pshallow_{\dl} (n) \sim 
\begin{cases}
\dl|n|^{2} & \mbox{if $n \ges \frac1\dl$}\\
|n|^{3} & \mbox{if $n \ll \frac1\dl$},
\end{cases}
\end{align}

\noi
where $\pshallow_{\dl} (n) $ is defined in \eqref{dispersionP}.
 As indicated by \eqref{eq:slow0}, in the high-frequency regime where $|n|\ges \frac1\dl$, the linear dispersion is dominated by $|n|^{2}$. 
It should also be noted that the scaled gILW equation \eqref{ILW3} only converges to the gKdV equation \eqref{KDV} when the frequency is fixed. 
 Thus, utilizing the same method as in this paper, there is no potential for improving the regularity even after scaling.
 Moreover, 
 when we construct the uniform bound over solutions in the shallow-water regime, some extra case-by-case analysis is needed to obtain 
 an uniform lower bound on the resonance function, see  Lemma \ref{lem_res2}.

%
   
   Let $v_\dl$ to be the solution of \eqref{ILW3}.
When examining the convergence of the scaled gILW solutions, 
we apply the same perturbative approach as in the deep-water situation. Again,
when we control the nonlinear perturbation,
the energy-type estimate necessitates the control of
$
\| v_{\g} \|_{N_{T}^{s,\dl}}
\les  \|v_{\g} \|_{N^{s,\g}_{T}}
$ (see Lemma \ref{LEM:ca3b}).
However, in the 
shallow-water case, for different $\g$ and $\dl$,
the perturbation of the dispersions is
\begin{align*}
\jb{ \tau - \pshallow_{\dl}(n) } \les 
\jb{\tau - \pshallow_{\g}(n)}  +   \jb{n}^{3}.
\end{align*}

\noi
The discrepancy between the symbols representing dispersion is now of the order $O(n^3)$. 
The naive try of the way we did in the deep-water case is no longer sufficient to absorb the third-order derivatives. 
To tackle this challenge, a frequency cutoff is introduced on the initial data and the frequency truncated equation \eqref{Tcay} is considered. The frequency truncation is applied to both the nonlinearity and the initial data. As a result,
the decay property of $h(n,\dl)$, as demonstrated in Lemma \ref{LEM:p1d}, can then be fully utilized to balance the term with third-order derivatives. Finally, in conjunction with the uniform continuity of the solution map, we obtain our convergence result in the shallow-water limit.
For more discussion, we refer to Subsection \ref{SUB:limit2}.

The convergence of dynamics of the ILW-type can be shown for the different initial data.
Specifically, the following corollary shows that convergence can be achieved with the addition of one convergence assumption regarding the different initial data. To demonstrate this concept, we provide the following statement as the example. This general principle can be extended to all of the convergence results outlined in Theorems \ref{THM:1} and \ref{THM:2}.

\begin{corollary}[Convergence with respect to the different initial data]
	\label{COR:4}
Consider the initial data $u_{\dl,0}$ and $u_{\infty,0}$ for the ILW equation \eqref{ILW1} and the BO equation $($\eqref{gBO} with $k=2$$)$, respectively. Let us assume that $u_{\dl,0}, u_{\infty,0} \in H^{s}(\M)$ for $s> \frac{1}{2}$, where $\M=\R \textup{ or }\T$, and that they satisfy:
	\begin{align*}
	\lim_{\dl\to\infty }\| u_{\dl, 0}- u_{\infty,0} \|_{H^{s}(\M)} =0.
		\end{align*}
Then, for any $0<T<1$,  we have the following convergence results:
	\begin{align*}
	\lim_{\dl \to \infty}
	\| u_\dl- u_{\infty} \|_{C([0,T]; H^s(\M))}=0,
	\end{align*} 
	
	\noi
where $u_{\dl}$ denotes the solution of the ILW equation \eqref{ILW1} with initial data $u_{\dl,0}$, and $u_{\infty}$ denotes the solution of the BO equation  with initial data $u_{\infty,0}$.
Moreover, $T=T(u_{\infty,0})$  depends only on the BO initial data, which is $\dl$ independent.
	
\end{corollary}

\begin{remark}
\rm
\label{RM:MD}

In our study, we will adopt the ungauged method established in \cite{MV15, MT} to obtain uniform control (in $\dl$) over the gILW solutions.
Such the ungauged approach was introduced to study the unconditional well-posedness of the dispersion generalised equation with rough initial data.
 An alternative ungauged approach can be found in \cite{KS21}.
 For an ungauged approach,
it is possible to reach the regularity $s=\frac12$ for the ILW convergence, where the unconditional well-posedness is not known
 as seen in the appendix of \cite{MV15}. 
 Due to the  strong low-high frequency interactions,
the next challenging problem will be
  achieving convergence in $C_TH^{s}$ with regularity $s<\frac12$. 
  In the study of the BO equation, the gauge transform developed by Tao \cite{Tao04} 
  enable us to study the BO equation with $L^2$ initial data, see for instance \cite{MP12, IK07}.
Therefore, one possible approach is to apply frequency dependent renormalisation method introduced in \cite{HIKK},
such method is analogous to the Tao's gauge transform but work well  for the BO equation with generalises dispersion.
As previously mentioned,   the convergence of gILW Gibbs dynamics (including $k=2$)
was constructed in \cite{LOZ} (which lacks uniqueness).
In particular, the support of Gibbsian data is in $H^{-\eps}(\T)\setminus L^2(\T)$ for $\eps>0$. 
Thus,  it is of challenging and interesting to reach the convergence results at the same level of the Gibbsian initial data and obtain the 
strong uniqueness statement.

%
%
%
%

 On the other hand,
 the recent breakthroughs by G\'{e}rard, Kappeler, and Topalov \cite{GKT} in exploring the complete integrability of the BO, they showed
  BO on the torus is globally  well-posed in $H^s(\T)$ for any $s>-\frac12$.
 It is therefore natural to consider the low regularity ILW convergence problem via  its  complete integrability.
 On the other hand, inspired by a series works of      Tzvetkov-Visciglia
\cite{TV1, TV2, TV3}, in \cite{CLOZ} the authors study the convergence of ILW dynamics at statistical equilibria, by constructing the corresponding dynamics of infinite sequence of weighted gaussian measures (associated to the conservation laws at $H^1$-level and above).

\end{remark}

%
%
%
%
%

\begin{remark}
	\label{RM:GWP}
	\rm

For a fixed $\delta > 0$, the Hamiltonian of the scaled gILW equation has been shown to belong to the $H^{\frac{1}{2}}$ space, as demonstrated in \cite{LOZ}.
Therefore, in general, it is not possible to extend the results of our convergence results to cover the entire time domain.
However, if we make the assumption that $s \geq 1$, we can apply the global well-posedness results of \cite{MT} to demonstrate that global convergence is attainable in certain scenarios. 
For further information on this topic, we direct the reader to the aforementioned publication \cite{MT}.

 If we focus solely on the ILW equation, it is widely recognized that this equation possesses an infinite number of conservation laws. 
 In particular, we have the following 
  $H^1$-level quantity:
 	\begin{align*}
	I_2(u)&:=\int
	\Big(
	\frac14 u^4
	+\frac34 u^2 \mathcal G_\dl \dx u + \frac18 (\dx u)^2+ \frac38 (\mathcal G_\dl \dx u)^2
	+\frac{3}{8\dl}  u\mathcal G_\dl \dx u \Big) dx.
	\end{align*}
	Therefore, this $H^{1}$-invariant quantity extends Theorem \ref{THM:1} and Theorem \ref{THM:2}
	globally-in-time.

\end{remark}

 \begin{remark}
 \rm
 \label{rem2} 
 
Our solutions are understood as distributional solutions.
Namely, for any test function $ \phi\in C_c^\infty((-T,T)\times \M) $, 
the following holds
  \begin{equation}
  \label{solu}
  \int_0^\infty \int_{\M}
   \Big(  (\phi_t + \G \, \phi_{xx} )u +  \phi_x u^k   \Big) \, dx \, dt +\int_{\M} \phi(0,\cdot) u_0 \, dx =0.
\end{equation}
%
%
%
%
%
Note that for  $u\in L^\infty([0,T];H^s(\M))$ with $s>\frac12$, 
$u^k$ is well-defined and belongs to $L^\infty([0,T];H^s(\M))$ for $k\geq 2$. 
Therefore,  \eqref{solu}  forces $\partial_t u \in L^\infty([0,T];H^{s-2}(\M))$
and ensures that \eqref{gILW1}  is satisfied in  $ L^\infty([0,T];H^{s-2}(\M))$.
 In particular, $u \in C([0, T ]; H^{s-2}(\M))$ and  \eqref{solu} forces the initial condition $u(0) = u_0$.
  Note that,
since $u \in L^\infty([0, T]; H^s(\M))$, this actually ensures that $u \in C_w([0, T ]; H^s(\M))$ and
 $u$ is in $C([0,T];H^\theta(\M))$ for any $\theta<s$.
This property also implies that $u$ satisfies  Duhamel formula associated with \eqref{gILW1}.


 \end{remark}

\begin{remark}\rm
	\label{GHW}

Our argument is applicable to both $\mathbb{R}$ and $\mathbb{T}$.
In comparison to the arguments in \cite{GW, HW}, they rely heavily on the local smoothing property, which is not available on $\mathbb{T}$. 
Lastly, we would like to point out that
it is possible to replace the nonlinearity of \eqref{gILW1}  by  $f(u)$
such that
$ f:\R\to\R$  is a real analytic function with an infinite radius of convergence. Namely, we have $ f\in C^\infty $ and satisfies
$	f(x)=\sum_{n=0}^{\infty} \frac{f^{(n)}(0)}{n!} x^n $	for  all $x\in \R$.
It is clear that any polynomial function, exponential functions as $ e^x$, $ \sin(x)$,  $\cos(x) $, and their products or compositions are also in this class, see \cite[Remark 1.3]{MT}.

\end{remark}

\section{Preliminaries}\label{SECT:Prelim}

In this section, we first introduce the necessary notations. Then, we will examine the basic behaviors of the dispersion terms in the (generalised) ILW equation and the scaled  (generalised)  ILW equation.
 Finally, we will introduce the function spaces used in this paper and their well-known properties.

\subsection{Notations} 
\label{SEC:nota}

For $A,B > 0$, we use $A\les B$ to mean that
there exists $C>0$ such that $A \leq CB$.
By $A\sim B$, we mean that $A\les B$ and $B \les A$.
Moreover, we denote $A\ll B$,
if there is some small $c> 0$,
such that $A \leq cB$.

For two non negative numbers $a,b$, we denote 
$a\vee b:=\max\{a,b\}$ and $a\wedge b:=\min\{a,b\}$.
We also write $\jb{\cdot}=(1+|\cdot|^2)^{1/2}$
for the Japanese bracket.

Given a function $u(t,x)$ on $\R\times \M$, we
use $\ft{u}$ and $\F (u)$
to denote the space Fourier transform of $u$ given by
\[ \ft{u}(k) = \int_{\M} e^{- ik x} u(t, x) \, dx
\qquad \qquad
\text{for } k\in \ft{\M}
.\]

\noi
In the remainder of this paper, we will primarily focus on the notation on $\T$ (i.e., $n \in \Z \setminus \{ 0\}$). For any $s \in \R$, we define $D^s f$ through its Fourier transform:
\begin{align*}
	\ft{D^s f}(n)=|n|^s\ft f(n).
\end{align*}

Let $\eta \in C_0^\infty(\R)$ be a even smooth non-negative cutoff function 
supported on $[-2,2]$ such that $\eta  \equiv 1$ on $[-1,1]$.
We define $\phi$ by $\phi(n)=\eta(n)-\eta(2n)$,
and 
set  $\phi_{2^{k}}(n) = \phi(2^{-k}n)$ for $k\in \Z$.
Namely,
$\phi_{2^{k}}$ is supported on $ \{  2^{k-1} \leq |n | \leq 2^{k+1} \}$.
By convention, we denote
$
\phi_1(n)=\eta(2n )$.

Let $\ZP = \Z \cap [0,\infty)$.
Given a (non-homogeneous) dyadic number $N \in 2^{\ZP}$,
we replace the above definition by $\phi_{N}$ for $N\geq 1$. 
Then, we have 
$
\sum^{\infty}_{  N= 1}
\phi_{N} =1$.
	We notice that
	$ \supp(\phi_N)\subset \{ N/2\le |n|\le 2N \}$
	for $ N\geq2 $ and if $N=1$, 
	$ \supp(\phi_1)\subset\{|n|\le 1\}.$
	Let $P_N$ be the (non-homogeneous) Littlewood-Paley projector onto the frequencies $\{ n \in \Z : |n| \sim N \}$, such that $\ft{P_N u} = \phi_N \ft{u}$. Then, we have $f = \sum_{N \geq 1} P_N f$. Additionally, we define $P_{\geq N} = \sum_{K \geq N} P_K$ and $P_{\leq N} = \sum_{K \leq N} P_K$.

Similarly, we also decompose the modulation function $(\tau - \pdeep_\dl(n))$ or $(\tau - \pshallow_\dl(n))$, depending on the context (see \eqref{dispersionP}), using the Littlewood-Paley projector $Q_L$, where $L$ is a dyadic number. 
We have $\ft{Q_L u} = \psi_L(n,\tau) \ft{u}$, where $\psi_L = \phi_L(\tau - \pdeep_\dl(n))$ or $\psi_L = \phi_L(\tau - \pshallow_\dl(n))$.

\subsection{Dispersion relation}
\label{SEC:DR}

In this subsection, we will review the properties of the dispersion relation associated with the gILW equations \eqref{gILW1}
and the scaled gILW equations \eqref{ILW3}. To start, we will remind the reader that the $\G$ operator is
\begin{align*}
\G=-\coth(\dl \dx)-\frac1\dl \dx^{-1} 
\end{align*}

\noi
and it
is understood as the Fourier multiplier defined by,
\begin{equation*}
\ft{\G }(n)
:=-i \Big( \coth( \delta n )-\frac{1}{ \delta n} 
\Big)
\qquad
\text{for $ n \in \Z \setminus \{0\} $} .
\end{equation*}

\noi
We use $\pdeep_{\dl}$ and $\pshallow_{\dl}$ to denote the linearized dispersion relations of the ILW-type equations and scaled ILW-type equations, respectively. These dispersion relations have the following forms:
\begin{align}
\begin{aligned}
\pdeep_{\dl}(n) &= n^{2} \Big( \coth( \delta n )-\frac{1}{ \delta n}  \Big),
\qquad \quad
\pshallow_{\dl}(n) &= \frac3\dl n^{2} \Big( \coth( \delta n )-\frac{1}{ \delta n}  \Big).
\end{aligned}
\label{dispersionP}
\end{align}

\noi
We observe that $\coth(\cdot)$ plays a crucial role in the expression \eqref{dispersionP}. In the following, we will collect some known results regarding the properties of $\coth(\cdot)$ by using the expansion formula.

\begin{lemma}[\cite{ABFS}~Lemma 8.2.1]
	\label{LEM:p1d}
	Let $\dl>0$ and for
	all $n\in \Z$,
	then we have
	\begin{equation*}
	n \coth( \dl n)=
	\frac{1}{\dl}
	+\frac{1}{3} \dl n^2
	-\frac{1}{3} n^2 h(n,\dl),
	\end{equation*}

	\noi
	where the remainder $h(n,\dl)=\sum_{k=1}^{\infty}   \frac{2\dl^3n^2}{k^2\pi^2 (k^2\pi^2 +\dl^2n^2)}$ satisfies the following conditions:

	\smallskip
	\noi  
	\hspace{1mm}
	{\rm (i)} For any finite $N\in\N$,  we have
	\begin{align*}
	\max_{|n|\leq N}  \big\|  h(n,\dl) \big\| 
	\les_N \dl^3.
	\end{align*}

	\smallskip
	\noi  
	 \hspace{1mm}
	{\rm (ii)} 
	There is some absolute constant $C_0$ such that for any $n\in\Z$,
	\begin{align*}
	|h(n,\dl)| \leq C_0 \dl .
	\end{align*}

	\smallskip
	\noi  
	 \hspace{1mm}
	{\rm (ii)} 
Let $2\leq \dl\leq \infty$. Then,
$ n\coth(  \dl n) \sim |n| $. In particular, we have
		\begin{equation*}
	-\frac{1}{\dl}
	+  |n|\leq   n \coth(  \dl n)
	\leq    \frac{1}{\dl}+   |n|.
	\end{equation*}

\end{lemma}

\begin{proof}
	The proof can be seen in
	\cite[Lemma 8.2.1]{ABFS} and \cite[Lemma 4.1]{ABFS}.
The essential idea is using
	the Mittag-Leffler expansion \cite{Bro} of $\coth(z)$ such that
$	z\coth(z)=1+\sum_{k=1}^{\infty} 
	\frac{2z^2}{z^2+(k\pi)^2}$.

\end{proof}

\begin{remark}\rm
\label{RM:p1}
 
Lemma \ref{LEM:p1d} implies that ${h(n,\dl)}{\dl}^{-1}$ is uniformly bounded by some absolute constant $C$ for all $n \in \mathbb{R}$ and $\dl > 0$. Furthermore, for fixed $n$ or for $n$ in any bounded interval, we have a good decay in $\dl$ such that ${h(n,\dl)}{\dl}^{-1} = O(\dl^2)$ as $\dl \to 0$.

\end{remark}

%
%
%
%
%

We immediately have Corollary \ref{LEM:p2b} and Lemma \ref{LEM:p12}.

\begin{corollary}[\cite{LOZ}~Lemma 2.1] 
	\label{LEM:p2b}

	Let $ K_{\dl}  := n\coth(\dl n) -\frac1\dl $
	Then, for any $\dl > 0$, we have 
	\begin{align}
	\max \Big( 0, |n|- \frac 1\dl\Big)\leq K_\dl(n)=n \coth(\dl n) -\frac{1}{\dl} \leq  |n| ,
	\label{ub5b}
	\end{align}
	
	\noi
	where the above inequalities are strict for $n\neq 0$.
	In particular, for $\dl\geq 2$ we have
	\begin{align}
	K_\dl(n) \sim |n|
	\label{HH1ab}
	\end{align}
	
	\noi
	for any $n \in \Z^*$.
	Furthermore, for each fixed $n \in \Z^*$, 
	$K_\dl(n)$ is strictly increasing in $\dl \ge 1$
	and converges to  $|n|$ as $\dl \to \infty$.
	
\end{corollary}

\begin{lemma}[\cite{LOZ}~Lemma 2.3]
	\label{LEM:p12}
	
	Let $L_\dl(n)=\frac3\dl K_{\dl}(n)$.
	The following statements hold.
	
	\begin{itemize}
		\item
		[{\rm (i)}] $0 <  L_\dl(n) <  n^2$ for any $\dl > 0$ and $n\in\Z^*$.

		\smallskip
		
		\item[{\rm (ii)}] For each $n\in\Z^\ast$, 
		$L_\dl(n)$ increases  to $n^2$ as $\dl\to 0$.

		\smallskip
		
		\item[{\rm (iii)}] 
		We have
		\begin{align*}
		L_\dl(n)\ges 
		\begin{cases}
		n^2, & \text{if } \dl |n| \les 1, \\
		|n|, & \text{if $\dl |n| \gg 1$ and $\dl \les 1$}.
		\end{cases}
		\end{align*}
		
		\noi
		In particular, the following uniform bound holds\textup{:}
		\begin{align*}
		\inf_{0 < \dl\les 1} L_\dl(n) \ges |n|
		\end{align*}
		
		\noi
		for any $n \in \Z^*$.

	\end{itemize}

\end{lemma}

\begin{lemma}[\cite{GW}~Lemma 3.1]
	\label{LEM:GW31}
	Let $\dl>0$ and $\pdeep_{\dl}(n)=n  (n\coth(\dl n) -\frac1\dl  ) $.
	Then, we have  the following statements:
	\begin{align*}
	\begin{cases}
	| \pdeep_{\dl}(n) | \sim |n|^{2}, 	\quad
	| \partial_n \, \pdeep_{\dl}{\dl}(n) | \sim |n|,	\quad
	| \partial^{2}_n \, \pdeep_{\dl}{\dl}(n) | \sim 1;	\quad
	\text{when}\, \,  |n|\ges \frac1\dl.\\
	| \pdeep_{\dl}(n) | \sim \dl |n|^{3},
	\quad
	| \partial_n \, \pdeep_{\dl}{\dl}(n) | \sim \dl |n|^{2},
	\quad
	| \partial^{2}_n \, \pdeep_{\dl}{\dl}(n) | \sim \dl |n|;
	\quad
	\text{when}\,\, |n|\les \frac1\dl.
	\end{cases}
	\end{align*}

\end{lemma}

\begin{remark}
	\rm
\label{RM:odd}

We now observe that 
\begin{align*}
\begin{aligned}
\pdeep_{\dl}(n)
&=n \Big(n\coth(\dl n) -\frac1\dl \Big )   \\
&=nK_{\dl}(n)
\in C^{1}(\R)\cap C^{2}(\R\setminus \{0\})
\end{aligned}
\end{align*}

\noi
and
\begin{align*}
\begin{aligned}
\pshallow_{\dl}(n)
&=\frac3\dl n \Big(n\coth(\dl n) -\frac1\dl \Big )   \\
&=\frac3\dl n K_{\dl}(n) =n L_{\dl}(n)
\in C^{1}(\R)\cap C^{2}(\R\setminus \{0\})
\end{aligned}
\end{align*}

\noi
are   real-valued odd functions. 	
For any fixed $\dl>0$, $\pdeep_{\dl}(n)$ and $\pshallow_{\dl}(n)$ satisfy the conditions in \cite[Hypothesis~1]{MT} as stated in \cite[Remark 1.2]{MT}.

\end{remark}

The resonance functions of the (scaled) gILW equations  are a result of the multi-linear interaction due to the nonlinearity. This interaction is referred to as non-resonant if the resulting frequency of multiple frequencies is large, and as resonant otherwise. In the non-resonant case, if the resonance function has a ``good" lower bound, then in Bourgain's Fourier restriction norm method, the modulation function provides derivative gain to balance the derivative loss in the nonlinearity.
In the following, we will study the properties of these resonance functions.
Before we proceed, we have the following definition:
\begin{definition}
\label{DEF:res}
	Let $j\in\N$ and  $(n_1,\dots,n_{j+1})\in\Z^{j+1}$.
	
	\begin{itemize}
		\item
		[{\rm (i)}] 
For any $2 \leq \dl< \infty$. We define 
	$\Om^{(\rm d, \dl)}_j   (n_1,\dots,n_{j+1}):\Z^{j+1}\to\R$  to be
	\begin{equation*}
	\Od_{j}  (n_1,\dots,n_{j+1}):=\sum_{k=1}^{j+1}\pdeep_{\dl}(n_k).
	\end{equation*}

		\smallskip
		
		\item[{\rm (ii)}] 
For any $0 < \dl\ll 1$. We define 
	$\Om^{(\rm s, \dl)}_j(n_1,\dots,n_{j+1}):\Z^{j+1}\to\R$  to be
	\begin{equation*}
	\Os_{j}  (n_1,\dots,n_{j+1}):=\sum_{k=1}^{j+1}\pshallow_{\dl}(n_k)
	\end{equation*}

	\end{itemize}

\end{definition}

To simplify the notation, we will use the following shorthand for the resonance function:
\begin{align*}
\Od_{j}  (n_1,\dots,n_{j+1})  = \Od_{j} ( \wt n)
\end{align*}

\noi
Now, we will show the resonance functions $\Od_{j} (\wt n)$ and $\Os_{j} (\wt n)$ have a uniform lower bound. 

\begin{lemma}
	\label{lem_res1}
	Let
	$k\ge 1 $, and 
	$ (n_1,\dots,n_{k+2})\in\Z^{k+2} $
	such that
$	\sum_{j=1}^{k+2}n_j=0.$
	Moreover, let us further assume that 
	\begin{equation*}
	\begin{cases}
	|n_1|\sim |n_2|\ges |n_3|,
	\qquad \qquad  \qquad\qquad  \hspace{-1mm}\text{if} \quad k= 1 ;\\
	|n_1|\sim |n_2|  \ges  |n_3|\gg  k  \displaystyle \max_{j\ge4 }|n_j|,
	\qquad  \text{if} \quad k\ge 2.
	\end{cases}
	\end{equation*}

	\noi
	Then, there exists some $n_0>0$ such that the following statements hold.

	\begin{itemize}
		\item
		[{\rm (i)}] 
Let $2 \leq \dl< \infty$. Then, for 
$ \displaystyle |n_1|\gg 
	\max_{ 0\leq n \leq n_0  } \big| \partial_n \, \pdeep_{\dl} (n)  \big| $,
 we have
	\begin{align}
	|\Od_{k+2}  ( \wt n   )|\ges  |n_3||n_1|.
	\label{Rd1}
	\end{align}

		\smallskip
		
		\item[{\rm (ii)}] 
Let $0< \dl<1$. 
Then, 	for
	$ \displaystyle |n_1|\gg 
	\max_{ 0\leq n \leq n_0  } \big| \partial_n \, \pshallow_{\dl} (n)    \big| $,
we have
	\begin{align}
	|\Os_{k+2}  (  \wt n)|\ges  |n_3||n_1|.
	\label{Rd2}
	\end{align}

	\end{itemize}

\end{lemma}

\begin{proof}

	The proof of the result is established based on Lemma 4.4 in \cite{MT}. 
	The key aspect of our analysis in equations (\ref{Rd1}) and (\ref{Rd2}) is the uniformity with respect to $\delta$, which is achieved through the uniform lower bound $|\partial_n \pdeep_{\delta}(n)|, |\partial_n \pshallow_{\delta}(n)| \geq |n|$.
	This uniformity result is a direct consequence of Lemmas \ref{LEM:p12} and \ref{LEM:GW31}, as well as the definition provided in Remark \ref{RM:odd}. Therefore, for the sake of brevity, the proof is not included here.

\end{proof}

\begin{lemma}
	\label{lem_res2}
	Let
	$ k\ge 2 $, 
	and $(n_1,\dots,n_{k+2})\in\Z^{k+2}$
	such that $
	\sum_{j=1}^{k+2}n_j=0.
	$
	Moreover, let us further assume that 
	\begin{equation*}
	|n_1|\sim |n_2|\gg  |n_3|\ges |n_4|,
	\qquad \text{if} \quad k= 2 ;\\
	\end{equation*}

	\noi
	for $ k\geq 3$ and 
	$|n_3+n_4|\gg k  \displaystyle \max_{j\ge5} |n_j|$
	we assume that
	\begin{equation*}
	|n_1|\sim |n_2|\gg   |n_3|\ges  |n_4|.
	\end{equation*}

	\noi
	Then, there exists some $n_0>0$
	such that the following statements hold.

	\begin{itemize}
		\item
		[{\rm (i)}] 
Let $2 \leq \dl< \infty$. Then, for 
$ \displaystyle
	|n_1|\gg \max_{0\leq n \leq n_0 }  |  \partial_n \,\pdeep_{\dl}(n)   | $,
we have
	\begin{align}
	|\Od_{k+2}  ( \wt n)|\ges  |n_3+n_4||n_1|.
	\label{Rd3}
	\end{align}
	
		\smallskip
		
		\item[{\rm (ii)}] 
Let $0< \dl< 1$. Then, for 
$ \displaystyle
	|n_1|\gg \max_{0\leq n \leq n_0 }  |  \partial_n \,\pshallow_{\dl}(n)   | $,
	we have	
	\begin{align}
	|\Os_{k+2}  ( \wt n)|\ges  |n_3+n_4||n_1|,
	\label{Rd4}
	\end{align}
	
	\end{itemize}

\end{lemma}

\begin{proof}

The proof is similar to that of \cite[Lemma 4.5]{MT}. We will only discuss the case when $|n_3| \sim |n_4|$. In this case, \eqref{HH1ab} implies the uniform (in $\dl$) lower and upper bounds of $\pdeep_\dl (n)$. In particular, we have Lemma \eqref{LEM:GW31} for $\dl \geq 2$. Thus, \eqref{Rd3} follows in the same way as in \cite[Lemma 4.5]{MT}. However, to obtain \eqref{Rd4}, we need to consider different cases by dividing the frequency regimes according to $\frac{1}{\dl}$.

	\smallskip
	\noindent
	\textbf{Proof of \eqref{Rd4}}.
	\smallskip

	Let $k \geq3$ and $n_3\sim n_4$.
	Then,
we need to consider $n_{3}$ and $n_{4}$ have the same or different signs.
	If  $n_3 n_4\ge0$.
	Then, we can write $|n_3+n_4|=|n_3|+|n_4|$.
	Moreover, we have 
	\[
	|n_3|,|n_4|\gg k\max_{j\ge5}|n_j|.
	\]

	\noi
	Therefore, the same proof \cite[Lemma 4.5]{MT} implies \eqref{Rd4}.

When $n_3 n_4<0$.
	By using the mean value theorem, there exist $k_1,k_2\in\R$ satisfying 
\begin{align}
\label{eq:ex}
	|k_1|\sim |n_1|\sim |n_2|
	\qquad {\rm and}
	\qquad
	|n_4|\les| k_2|\les|n_3|
\end{align}

	\noi
	such that
	\begin{align}
	\begin{aligned}
	-\Os_{k+2} (\wt n)
	&=-(n_1+n_2) \,  \partial_n \, \pshallow_{\dl}(k_1)
	-(n_3+n_4)   \,  \partial_n \, \pshallow_{\dl}(k_2) -\sum_{j=5}^{k+2} \pshallow_{\dl}(n_j)  \\
	&=(n_3+n_4+\cdots+n_{k+2})
	\,    \partial_n \, \pshallow_{\dl}(k_1) 
	-(n_3+n_4)
	\, \partial_n \,  \pshallow_{\dl}(k_2)   - \sum_{j=5}^{k+2} \pshallow_{\dl}(n_j)
	\end{aligned}
	\label{Rdd1}
	\end{align}

	\noi
	where we used the property of $\pshallow_{\dl}(n)$ being an odd-function. Moreover, to see \eqref{eq:ex}, we notice
 $k_{1}$ is between $-n_{1}$ and $n_{2}$, 
		and 
		$k_{2}$ is between $-n_{3}$ and $n_{4}$.
		Since, $|n_{1}| \sim |n_{2}|\gg |n_{3} +\dots +  n_{k+2}|$
		and $ \sum_{j=1}^{k+2}n_{j}=0$,
		we have $ -n_{1}$ and $n_{2}$ must have the same sign. Thus 
		$|k_{1}|\sim |n_{1} | \sim |n_{2}|$.
		Moreover,
		this case, we are under the assumption that $n_{3}n_{4} <0$.
		Therefore, $|k_{2}| \sim |n_{3}| \sim |n_{4}|$.

	Next, it is enough to show 
	\begin{equation*}
	  \big| (n_3+n_4)  \, \partial_n \, \pshallow_{\dl}(k_2)  \big|
	\qquad
	{\rm and}
	\qquad
	\sum_{j=5}^{k+2}   \big| \pshallow_{\dl}(n_j) \big|
	\end{equation*}
	are negligible comparing to $ |n_3+n_4||n_1|$.
Here, we observe
when 
	$|k_2|\le n_0$, by our constraint we have
	\begin{equation*}
	| \partial_n \, \pshallow_{\dl}(k_2)  | \leq \max_{0\leq n\leq n_0 }|   \partial_n \, \pshallow_{\dl}(n)  |\ll   |n_1|.
	\end{equation*}
	
	\noi
But, if
	$|k_2|\ge n_0\ges \frac{1}{\dl}$, by Lemma \eqref{LEM:GW31} we have no {\bf uniform} in $\dl$ upper bound on 
 	$|  \partial_n \, \pshallow_{\dl}(k_2) |  $.
 	This is where 
	the direct application of  the proof of \cite[Lemma 4.5]{MT}  fails.

In order to obtain the uniform lower bound \eqref{Rd4}, we need more information on  $\pshallow_{\dl}(n)$.
	Following from \eqref{Rdd1}
 if we have claim:
	\begin{align}
	\label{Rddc}
	|\Os_{k+2} (\wt n)|
	\sim |n_{3}+n_{4}| | \partial \pshallow_{\dl}(k_{1}) |  \quad 
	\text{where} \quad
	|k_{1}|\sim |n_{1}|\sim |n_{2}|.
	\end{align}

	\noi
By using Lemma \ref{LEM:GW31}, \eqref{Rddc}  means that if we have the following  
	\begin{align}
	\label{Rddc1}
	\begin{cases}
	|\Os_{k+2} (\wt n)|  \sim |n_{3}+n_{4}| |n_{1}|^{2}
	\quad
	&\text{for} \quad
	|n_{1}|\les \frac1\dl,\\
	|\Os_{k+2} (\wt n)|  \sim  \frac1\dl |n_{3}+n_{4}| |n_{1}|
	\quad
	&\text{for} \quad
	|n_{1}|\ges \frac1\dl.
	\end{cases}
	\end{align}

	\noi
Hence,
claim \eqref{Rddc} and
 \eqref{Rddc1} with the condition $0<\dl\ll1$ imply that
	for any $n_{1}\in \Z^*$  we have \eqref{Rd4}.
Next, we prove claim \eqref{Rddc}.

	\smallskip
	\noindent
	\textbf{Case 1:} $|n_1|\les \frac1\dl$.
	\smallskip

	In this case, we have $|n_{j}| \les \frac1\dl$ for all $j\geq 1$.
	Then, we have the following 4 estimates:
	\begin{itemize}
		\item
		$ |n_{3}+n_{4}| | \partial \pshallow_{\dl}(k_{1}) |  \sim  |n_{3}+n_{4}|  |n_{1}|^{2} $
		
		\smallskip		
		\item
		$ |n_{5}+\dots + n_{k+2}|
		| \partial \pshallow_{\dl}(k_{1}) |  \sim  |n_{3}+n_{4}|  |n_{1}|^{2} $

		\smallskip		
		\item
		$ |n_{3}+n_{4}| | \partial \pshallow_{\dl}(k_{2}) |  \sim  |n_{3}+n_{4}|  |n_{3}|^{2}
		\ll   |n_{3}+n_{4}|  |n_{1}|^{2} $
		
		\smallskip		
		\item
		$\displaystyle \sum_{j=5}^{k+2}  |\pshallow_{\dl}(n_{j}) |\sim  \sum_{j=5}^{k+2}  |n_{j}|^{3}   \ll   |n_{3}+n_{4}|  |n_{1}|^{2} $

	\end{itemize}

	This completes the proof for this case.

	\smallskip
	\noindent
	\textbf{Case 2:} $|n_1|\ges \frac1\dl$.
	\smallskip

	In this case, we  have $|n_1|\sim |n_{2}|\ges \frac1\dl$.
	Then, Lemma \ref{LEM:GW31} implies 
	\begin{itemize}
		\item
		$ |n_{3}+n_{4}| | \partial \pshallow_{\dl}(k_{1}) |  \sim  \frac1\dl  |n_{3}+n_{4}|  
		|n_{1}|   \ges  |n_{3}+n_{4}| 
		|n_{1}|  $
		
	\smallskip	
		\item
		$ |n_{5}+\dots + n_{k+2}|
		| \partial \pshallow_{\dl}(k_{1}) |  
		\sim \frac1\dl |n_{5}+\dots + n_{k+2}|   |n_{1}|   \ll  \frac1\dl
		|n_{3}+n_{4}| 
		|n_{1}|     $

	\end{itemize}

	For the remaining terms, we need to consider
	cases depending on how big or small these frequencies are when compared to  $\frac1\dl$.
	\begin{equation*}
	|\partial \pshallow_{\dl}(k_{2}) | \sim
	\begin{cases}
	\frac1\dl |k_{2}|, \quad \text{if} \,\,  |k_{2}|\ges \frac1\dl;\\
	|k_{2}|^{2}, \quad \, \text{if} \,\,  |k_{2}|\les \frac1\dl.
	\end{cases}
	\end{equation*}

	\noi
	Since, we have $|k_{2}| \sim |n_{3}|\ll |n_{1}|$ and $|n_{1}| \ges \frac1\dl$. Then,
	\begin{itemize}
		\item
		when $|k_{2}| \ges \frac1\dl $,
		$|\partial \pshallow_{\dl}(k_{2}) | \sim \frac1\dl |k_{2}| \ll \frac1\dl |n_{1}| \sim |\partial \pshallow_{\dl}(k_{1})|$;
		
	\smallskip			
		\item
		when $|k_{2}| \les \frac1\dl $,
		$|\partial \pshallow_{\dl}(k_{2}) |^{2} \sim \frac1\dl |k_{2}| \ll \frac1\dl |n_{1}| \sim |\partial \pshallow_{\dl}(k_{1})|$.

	\end{itemize}

	\noi
	Therefore,
	we always have 
	$|\partial \pshallow_{\dl}(k_{2}) | |n_{3} +n_{4}| \ll |n_{3} +n_{4}| \frac1\dl |n_{1}|$.
	Next, for each $j\geq 5$, we have
	\begin{equation*}
	|\partial  \pshallow_{\dl}(n_{j}) | \sim
	\begin{cases}
	\frac1\dl |n_{j}|^{2}, \quad \text{if} \,\,  |n_{j}|\ges \frac1\dl;\\
	|n_{j}|^{3}, \quad \, \text{if} \,\,  |n_{j}|\les \frac1\dl.
	\end{cases}
	\end{equation*}

	\noi
	If $ |n_{j}|\les \frac1\dl$, then 
	\begin{align*}
	|\pshallow_{\dl}(n_{j})|\sim |n_{j}|^{3}&\ll  \frac1k  |n_{3} +n_{4}|  |n_{j}|^{2}\\
	&\ll \frac{1}{k\dl}  |n_{3} +n_{4}|  |n_{1}|.
	\end{align*}

	\noi
	If $ |n_{j}|\ges \frac1\dl$, then 
	\begin{align*}
	|\pshallow_{\dl}(n_{j})|\sim  \frac1\dl |n_{j}|^{2} \ll \frac{1}{k\dl}  |n_{3} +n_{4}|  |n_{1}|.
	\end{align*}

	\noi
	Hence, we can conclude  \eqref{Rddc}.
	For the case $k=2$, we can argue exactly as above.
\end{proof}

\subsection{Function spaces and their basic properties}

In this subsection, we introduce the function spaces and their properties. To start with, we present a sequence of positive numbers $\{\om_N\}_{N}$, which is an increasing sequence that depends on the dyadic number $N \in 2^{\ZP}$. This sequence of weights $\{\om_N\}_{N}$ is referred to as the frequency envelope in \cite[Section 5]{Tao04}. Its main purpose is to be useful in proving continuity with respect to the initial data, see Remark \ref{RM:cts}. This technique was first introduced in \cite{KT}. 
Additionally, we extract the following result from \cite[Lemma 4.6]{MT}, which will assist us in choosing our frequency envelope $\om_{N}$.

\begin{lemma}[\cite{MT}~Lemma 4.6]
	\label{LEM_ev}
Let $\kk> 1$, suppose the dyadic sequence $\{\om_N\}$ of positive numbers~satisfies 
\begin{align}
\label{eq_dydic}
\om_N\le \om_{2N}\le \kk\om_N
\quad\quad
\text{for}
\quad 
N\ge1,
\end{align}
	
	\noi
	and $\om_N\to\infty$ as $N\to\infty$.
	Then, for any $1<\kk'<\kk$, there exists a dyadic sequence $\{\wt{\om}_N\}$ such that
\begin{align*}
	\wt{\om}_N\le \om_N, 
	\quad
	\wt{\om}_N\le \wt{\om}_{2N}\le \kk'\wt{\om}_N
	\quad \quad
	\text{for}
	\quad
	N\ge 1
\end{align*}

	\noi
	and $\wt{\om}_N\to\infty$ as $N\to\infty$.
\end{lemma}

With the aid of Lemma \ref{LEM_ev}, for a given dyadic sequence ${\omega_N}$ of positive numbers, it is possible to choose $\kk \leq 2$. This in turn allows us to define a new dyadic sequence. Given two dyadic numbers $N$ and $M$ such that $1 \leq M \leq \lambda N$ for some $\lambda \geq 2$, we can use the inequality $\omega_{2N} \leq \kk \omega_N$ to deduce that:
\begin{equation*}
\frac{\om_M}{\om_N}
\les \kk^{\log_2 \lambda}
\les \lambda
\end{equation*}

\noi
which is uniformly in $\kk$.

In light of the preceding discussion on the dyadic sequence ${\omega_N}$ of positive numbers, we propose a slight modification to the definition of the $L^2$-based Sobolev spaces. For a given value of $s \geq 0$, we define the space $H_\omega^s(\mathbb{T})$ with the following norm:
\begin{align*}
\|u\|_{H_\om^s}
:=\Big(\sum_{N,\textup{dyadic} }\om_N^2 (1\vee N)^{2s}\|P_N u\|_{L^2}^2    \Big)^{\frac12}.
\end{align*}

\noi
One simple observation is by selecting $\omega_N \equiv 1$, we can recover the standard $L^2$-based Sobolev space. In other words, if we set $\omega_N \equiv 1$, then $H_\omega^s(\mathbb{T}) = H^s(\mathbb{T})$.

For a given range of values $1 \leq p \leq \infty$ and a positive time value $T > 0$, let $B_x$ be an arbitrary Banach space. To facilitate our analysis, we introduce the following shorthand notation:
\begin{align*}
L_t^p B_x :=L^p(\R;B_x)
\qquad
\text{\rm and}
\qquad
L_T^p B_x :=L^p([0,T];B_x)
\end{align*}

\noi
equipped with the norms
\begin{align*}
\|u\|_{L_t^p B_x}=\Big(\int_\R\|u(t,\cdot)\|_{B_x}^p dt    \Big)^{\frac1p}\qquad
\textrm{and}\qquad
\|u\|_{L_T^p B_x}=\Big(     \int_0^T\|u(t,\cdot)\|_{B_x}^p dt     \Big)^{\frac1p},
\end{align*}

\noi
respectively.
In the case where $p = \infty$, the physical space is modified to the space of essentially bounded measurable functions, equipped with the essential supremum norm.

As is typical in the low-regularity analysis of dispersive  PDEs, the Fourier restriction norm method plays a crucial role. This method was introduced in the publications by Bourgain \cite{B93a, B93b}. For given values of $s, b \in \mathbb{R}$, we define the space $X^{s,b,\delta}(\mathbb{T} \times \mathbb{R})$, denoted by $X^{s,b,\delta}$, as the completion of the test functions with respect to the following norm:
\begin{align}
\label{deX}
\begin{aligned}
\|u\|_{X^{s,b,\dl }(\T\times \R)}
&=\Big(\sum_{n=-\infty}^\infty \int_{-\infty}^\infty \jb{n}^{2s}\jb{\tau-\pdeep_{\dl}(n)}^{2b}|\ft{u}(\tau,n)|^2 d\tau    \Big)^{\frac12}\\
&=\|   \jb{n}^{s}\jb{\tau-\pdeep_{\dl}(n)}^{b}    \ft{u}(\tau,n)  \|_{\l^{2}_{n}   L^{2}_{\tau} }.
\end{aligned}
\end{align}

\noi
Similarly, we define $Y^{s,b,\dl}(\T\times \R)$ (=$Y^{s,b,\dl}$)  as a completion of the
test functions under the following norm:
\begin{align*}
\begin{aligned}
\|v\|_{Y^{s,b,\dl}(\T\times \R)}
&=\Big(\sum_{n=-\infty}^\infty \int_{-\infty}^\infty \jb{n}^{2s}\jb{\tau-\pshallow_{\dl}(n)}^{2b}|\ft{v}(\tau,n)|^2 d\tau    \Big)^{\frac12}\\
&=\|   \jb{n}^{s}\jb{\tau-\pshallow_{\dl}(n)}^{b}    \ft{v}(\tau,n)  \|_{\l^{2}_{n}   L^{2}_{\tau} }.
\end{aligned}
\end{align*}

\noi
We also use a slightly stronger space $X_\om^{s,b,\dl}$ with the norm
\begin{equation*}
\|u\|_{X_\om^{s,b,\dl}}
:=\Big(   \sum_{N}\om_N^2 (1\vee N)^{2s}\|P_N u\|_{X^{0,b}}^2  \Big)^{\frac12},
\end{equation*}

\noi
and
$Y_\om^{s,b,\dl}$ with the norm
\begin{equation*}
\|v\|_{Y_\om^{s,b,\dl}}
:=\Big(   \sum_{N}\om_N^2 (1\vee N)^{2s}\|P_N v\|_{Y^{0,b}}^2  \Big)^{\frac12}.
\end{equation*}

\noi
Moreover,
we define the function spaces $M^{s,\dl} $  and $M^{s,\dl}_\om $  in the following way.
\begin{equation*}
M^{s,\dl}:= L_t^\infty H^s\cap X^{s-1,1,\dl}
\qquad\qquad
M^{s,\dl}_\om:= L_t^\infty H_\om^s \cap X_\om^{s-1,1,\dl},
\end{equation*}

\noi
endowed with the natural norm
\begin{equation*}
\|u\|_{M^{s,\dl}  }=\|u\|_{L_t^\infty H^s}+\|u\|_{X^{s-1,1,\dl}}
\qquad \qquad
\|u\|_{M^{s,\dl}_\om}=\|u\|_{L_t^\infty H^s_\om}+\|u\|_{X^{s-1,1,\dl}_\om} .
\end{equation*}

\noi
In the same line as above, we 
define the function spaces $N^{s,\dl} $  and $N^{s,\dl}_\om $:
\begin{equation*}
N^{s,\dl}:= L_t^\infty H^s\cap Y^{s-1,1,\dl}
\qquad\qquad
N^{s,\dl}_\om:= L_t^\infty H_\om^s \cap Y_\om^{s-1,1,\dl},
\end{equation*}

\noi
endowed with the natural norm
\begin{equation*}
\|v\|_{N^{s,\dl}  }=\|v\|_{L_t^\infty H^s}+\|v\|_{Y^{s-1,1,\dl}}
\qquad \qquad
\|v\|_{N^{s,\dl}_\om}=\|v\|_{L_t^\infty H^s_\om}+\|v\|_{Y^{s-1,1,\dl}_\om} .
\end{equation*}

\noi
In addition, we can also consider the time-restricted versions of these spaces. Given a positive time value $T > 0$ and a normed space of space-time functions $B$, the restriction space $B_T$ consists of functions $u:(0,T) \times \mathbb{T} \to \mathbb{R}$ that satisfy:
\begin{align*}
\|u\|_{B_T}
:=\inf\{   
\|  \wt{u} \|_B  \, | \,    \wt{u} : \R\times \T\to\R\, ;
\wt{u}=u\,
\textrm{on}\,
(0,T)\times\T\}<\infty.
\end{align*}

\subsection{Uniform linear estimates}

The main result of this section is to establish $\delta$-independent version of short-time Strichartz estimates, 
which were first introduced in the work by Koch-Tzvetkov \cite{KT}.
The results of this section is particularly used to control the energy-type of estimate of gILW for $k\geq 2$, this type of argument can be seen in \cite{MT}.
If we only focus on $k=2$, we do not need this section to control the nonlinear interaction, see \cite{MV15}.
 Moreover, 
 it is important to note that these uniform estimates must be established separately for the shallow-water and deep-water regimes.

\begin{proposition}
	\label{PRO:stri}

Let $k \geq 2$.
Consider $u_0, v_0 \in H^s(\mathbb{T})$ for $s > \frac 12$ and 
$u,v \in C([0,T]; H_{\omega}^{s}(\mathbb{T}))$ satisfy  gILW  \eqref{gILW1} and  scaled gILW \eqref{ILW3}, respectively, 
on the interval $[0,T]$ with $0<T<1$.
	Then,  for
$\{ \omega_N \}$ be a dyadic sequence that satisfies \eqref{eq_dydic} with $\kk\geq 1$,
the following statements hold.

		\begin{itemize}
		\item[(i)]
Let  $2\leq \dl< \infty$. Then, we have
	\begin{align*}
	\Big(     \sum_{N, {\rm dyadic}} \om_N^4\|D_x^{s-\frac 18}P_N u\|_{  L^4([0,T];L^{4} (\T)  )  }^4    \Big)^{\frac 14}
	&\leq C  T^{\frac 18}   \|u\|_{L^{\infty}([0,T]; H_\om^s(\T) ) };\\
	\Big(  \sum_{N, {\rm dyadic}}  \|D_x^{\frac 13}P_N u\|_{L^3([0,T]; L^\infty (\T)  ) }^3   \Big)^{\frac 13}
	&\leq   C T^{\frac{5}{24}}    \|u\|_{L^\infty([0,T]; H^{\frac{17}{24}}(\T) )   }.
	\end{align*}

	\smallskip			
		\item[(ii)]
Let $0<\dl\ll1$.
Then, we have
	\begin{align*}
	\Big(     \sum_{N, {\rm dyadic}}  \om_N^4\|D_x^{s-\frac {1}{8}}P_N v\|_{    L^4([0,T];L^{4} (\T)  )   }^4    \Big)^{\frac 14}
	&\leq  \wt C T^{\frac 18}   \|v\|_{   L^{\infty}([0,T]; H_\om^s(\T) )   }; 
	\\
	\Big(     \sum_{N, {\rm dyadic}}
	\|D_x^{\frac 13}P_N v\|_{     L^3([0,T]; L^\infty (\T)  )   }^3    \Big)^{\frac 13}
	&\leq \wt C  T^{\frac{5}{24}}    \|v\|_{ L^\infty([0,T]; H^{\frac{17}{24}}(\T) )  }   .
	\end{align*}

	\end{itemize}

	\noi
	Here, the constant $C=C(\|u\|_{L^\infty_{T,x}})$  and $\wt C=\wt C(\|v\|_{L^\infty_{T,x}}) $ are  independent of $\dl$.
\end{proposition}

\begin{proof}

The proof follows from \cite[Lemma 3.5]{MT},
once we obtain uniform (in $\dl$) estimates of Lemmas \ref{lem_stri1} and \ref{LEM:str2}.	
	We shall skip the proof here.

\end{proof}

The first step in proving Proposition \ref{PRO:stri} is to establish the following $L^4$-Strichartz estimate, which was first introduced in \cite{B93a, B93b}. 
\begin{lemma}[Uniform $L^4$-Strichartz estimate]
	\label{lem_stri1}
	
	Let $u \in X^{0,\frac38,\dl} (\T\times\R)  $ and 
	$v \in Y^{0,\frac38,\dl}  (\T\times\R) $.
	Then, there exists a universal constant $C$ such that the following estimates hold.
	
		\begin{itemize}
		\item[(i)]
Let $2\leq \dl \leq \infty$.
Then, we have
	\begin{align*}
	\|u\|_{L^4(\R;L^{4}(\T))} \leq C\|u\|_{X^{0,\frac38,\dl} (\T\times\R) }.
	\end{align*}

	\smallskip	
		\item[(ii)]
Let $0<\dl\ll1$. Then, we have
	\begin{align*}
	\|v\|_{L^4(\R;L^{4}(\T))}  \leq C\|v\|_{Y^{0,\frac38,\dl} (\T\times\R) }.
	\end{align*}

	\end{itemize}

\end{lemma}

\begin{proof}
	
The proof of the result is
a direct consequence of Lemma \ref{LEM:strs1}.
This type of argument can be found in the Appendix of \cite{M07} for similar considerations.
\end{proof}

\begin{lemma}
	\label{LEM:strs1}
	
Consider $u,v$ belong in  $L^2(\mathbb{R}; \ell^2(\mathbb{Z}))$ 
to be real-valued functions, and let $N_1,N_2,M,\widetilde{M} \in 2^{\ZP}$ be dyadic numbers. Set $M = \min\{N_1,N_2\}$ and $\widetilde{M} = \max\{N_1,N_2\}$. Then, there exists a universal constant $C$ such that the following estimates hold.
	
		\begin{itemize}
		\item[(i)]
Let  $2\leq \dl\leq \infty$. Then,  we have
	\begin{align}
	\|(\psi_{N_1}u)*_{\tau,n}(\psi_{N_2}v)\|_{L_\tau^2 \l_n^2}
	\leq
	C M^{\frac12}  \wt M^{\frac14}
	\|\psi_{N_1}u\|_{L_\tau^2 \l_n^2}
	\|\psi_{N_2}v\|_{L_\tau^2 \l_n^2}.
	\label{eq:s1}
	\end{align}

	\smallskip	
			
		\item[(ii)]
Let $0<\dl <1 $. Then, we have
	\begin{align}
	\|(\psi_{N_1}u)*_{\tau,n}(\psi_{N_2}v)\|_{L_\tau^2 \l_n^2}
	\leq 
	C M^{\frac12}  \wt M^{\frac14}
	\|\psi_{N_1}u\|_{L_\tau^2 \l_n^2}
	\|\psi_{N_2}v\|_{L_\tau^2 \l_n^2}.
		\label{eq:s2}
	\end{align}
	
	\end{itemize}

\noi
Here, $\psi_{N}$ represents the projection onto the modulation function.

\end{lemma}

\begin{proof}

The proof follows from\cite[Lemma 3.2]{MT} and it applies to both shallow-water and deep-water regimes.
In particular, we see from Remark \ref{RM:odd} that both 
	\begin{align*}
\pdeep_{\dl}(n)=nK_\dl(n) 
\qquad \qquad
\pshallow_{\dl}(n)= nL_\dl(n)
	\end{align*}
satisfy  \cite[Hypothesis~1]{MT} with some $n_{0}>0$.
Moreover, Lemma \ref{LEM:p12} and 
Corollary \ref{LEM:p2b} provide uniform (in $\dl$) lower bounds on $K_\dl(n)$ and $L_\dl(n)$ such that
\begin{align*}
|K_\dl(n)|, |L_\dl(n)| \ges |n| 
\end{align*}
where they are defined in \eqref{ub5b} and Lemma \ref{LEM:p12}. 
These uniform lower bound are the crucial step in applying the following counting~\cite[Lemma 2]{ST01}:
Let $I$ and $J$ be two intervals on the real line and $g\in C^1(J;\R)$.
	Then, we have
	\begin{align*}
	\# \{x\in J\cap\Z;g(x)\in I\}\le\frac{|I|}{\inf_{x\in J}|\partial_x g(x)|}+1.
	\end{align*}
Then,  the argument as in \cite[ Lemma 3.2]{MT} applies in our situations.
Moreover, estimates \eqref{eq:s1} and \eqref{eq:s2} are independent of $\dl$.

\end{proof}

Finally, Lemma \ref{lem_stri1} enables the establishment of uniform linear estimates. The proof of this result is based on  \cite[Lemma 2.1]{MR09}; additional references can be found in \cite{MT}.
Let $\Sdt    $ be the linear propagators of the (generalised) ILW equation defined as
\begin{align}
\label{LiP}
S_{\dl}^{{\rm (d)}}(t)
= e^{-t\G \dx^{2}}.
\end{align}
Similarly, for the scaled (generalised) ILW equation  we define 
$S_{\dl}^{{\rm (s)}}(t)
= e^{-\frac3\dl t \G \dx^{2}}$.

\begin{lemma}
	\label{LEM:str2}
	Let $T>0$,
	any $u,v\in L^2(\T)$, and 
	$S_{\dl}^{{\rm (s)}}(t), S_{\dl}^{{\rm (d)}}(t)$ as defined in \eqref{LiP}.
	Then, there exists a universal constant $C$ such that the following estimates hold.	
		\begin{itemize}
		\item[(i)]
Let $ 2\leq \dl \leq \infty$. Then,
	we have
	\begin{equation*}
	\|\Sdt u\|_{L^{4}([0,T];L^4(\T))}
	\le CT^{\frac18}\|u\|_{L^2(\T)}.
	\end{equation*}

	\smallskip	
			
		\item[(ii)]
Let $0<\dl< 1$. Then,
	we have
	\begin{equation*}
	\|\Sst v\|_{L^{4}([0,T];L^4(\T))}
	\le CT^{ \frac18}    \|v\|_{L^2(\T)}.
	\end{equation*}

	\end{itemize}

\end{lemma}

The following difference estimate follows from \cite[Corollary 3.6]{MT}.
\begin{corollary}
	\label{COR:stri}
Let $k \geq 2$,	$s>\frac 12$ and
	$0<T<1$. The following two situations are assumed:
\begin{itemize}
	\item[\textup{(i)}]
Let  $2 \leq \dl< \infty$. Consider 
 $u^{(1)},u^{(2)}$ belong in $ C([0,T]; H^{s}(\T))$  and 
		satisfy gILW \eqref{gILW1}
		with $u^{(1)}_0,u^{(2)}_0\in H^{s}(\T)$ on $[0,T]$.
		
		\smallskip
			\item[\textup{(ii)}]
Let $0<\dl< 1$.
Consider $v^{(1)},v^{(2)}$ belong in $C([0,T]; H^{s}(\T))$  and satisfy scaled gILW \eqref{ILW3}
with $v^{(1)}_0,v^{(2)}_0\in H^{s}(\T)$ on $[0,T]$.	
	\end{itemize}

\noi
Additionally, let us define $w_1:=u^{(1)}-u^{(2)}$ and $w_2:=v^{(1)}-v^{(2)}$.
Then,  there exist    constants $C_1 = C(u^{(1)},u^{(2)})$ and $C_2 = C(v^{(1)},v^{(2)})$
such that  the following estimates hold for $j = 1, 2$.
	\begin{align*}
	&\Big(\sum_{N, \textup{dyadic }} 
	\big[
	(1\vee N)^{s-\frac98}\|P_N w_j\|_{L^4([0,T]; L^{4}(\T) )}
	\big]^4   \Big)^{\frac14}
	\leq C_j  T^{\frac18}  \|w_j\|_{L^{\infty}([0,T]; H^{s-1} (\T)  };\\
	&\Big(\sum_{N, \textup{dyadic }}
	\big[
	(1\vee N)^{-\frac{5}{12}}
	\|P_N w_j\|_{L^3([0,T]; L^4(\T) ) }\big]^3
	\Big)^{\frac 13}
	\leq C_j T^{\frac{5}{24}}  
	\|w_j\|_{L^{\infty}([0,T];  H_x^{-\frac{7}{24}} (\T) )    }.
	\end{align*}
 	
\end{corollary}

\subsection{Uniform energy estimates}

In this subsection, we will present the crucial energy estimates that are necessary to ensure that all estimates are uniformly in $\delta$. 
For simplicity, we write $u_\delta = u$ and $v_\delta = v$ in this section.

The following lemma is a key tool for achieving unconditional uniqueness, and it utilizes the fact that for $s > \frac{1}{2}$, the solutions $u$ of the gILW equation \eqref{gILW1} and $v$ of the scaled gILW equation \eqref{ILW3} also satisfy the Duhamel formulation.

\begin{lemma}
	\label{lem1}
Let $k \geq 2$.
 Consider $u_0, v_0 \in H_\om^s(\mathbb{T})$ for $s > \frac 12$ and 
$u,v \in L^\infty([0,T]; H_\om^s(\mathbb{T}))$
 to be solutions of  gILW  \eqref{gILW1} and  scaled gILW  \eqref{ILW3}, respectively.
 Then,  for $\{ \omega_N  \}$ be a dyadic sequence that satisfies \eqref{eq_dydic} with $1\leq\kk\leq2$,
the following statements hold.

	 \begin{itemize}
	 \item[(i)]
Let  $2\leq \dl\leq \infty$. Then, $ u\in M^{s,\dl}_{\om,T}$ and  we have
	 	\begin{equation}
	 	\label{eq:ms}
	 	\|u\|_{M^{s,\dl}_{\om,T}}  \les \|u\|_{L^\infty_T H^s_\om} +C(\|u\|_{L^\infty_{T,x}}) \|u\|_{L^\infty_T H^{s}_\om}\;.
	 	\end{equation}

	 	\noi
	 	Moreover, for $j=1,2$.
	 	Let  $u^{(j)} \in L^\infty([0,T]; H^s(\T) $ to be solutions
	 	of gILW \eqref{gILW1} 
	 	with  initial data
	 	$u_0^{(j)} \in  H^s(\T) $.
	 	Then,
	 	the following holds
	 	\begin{align}
	 	\begin{aligned}
	 	\label{eq:mdf}
	 	\|u^{(1)}  -  u^{(2) }\|_{ M^{s-1, \dl}_{T} }  &\les 
	 	\|u^{(1)}   -  u^{(2)}   \|_{ L^{\infty}_{T} H_{x}^{s-1} } \\
	 	&\quad+  C\big( \|   u^{(1)}    \|_{ L^{\infty}_{T} H_{x}^{s} }  +    \| u^{(2)}    \|_{ L^{\infty}_{T} H_{x}^{s} } \big) 
	 	\| u^{(1)}-  u^{(2)}    \|_{L^{\infty}_{T} H_{x}^{s-1} }.
	 	\end{aligned}
	 	\end{align}

		\smallskip
	 	
	 	\item[(ii)]
Let $0<\dl< 1$. Then, $ v\in N^{s,\dl}_{\om,T}$ and we have
	 	\begin{equation*}
	 	\|v\|_{N^{s,\dl}_{\om,T}}  \les 
	 	\|v\|_{L^\infty_T H^s_\om} +C(\|v\|_{L^\infty_{T,x}}) \|v\|_{L^\infty_T H^{s}_\om}\;.
	 	\end{equation*}

	 	\noi
	 	Moreover, for $j=1,2$.
	 	Let  $v^{(j)} \in L^\infty([0,T]; H^s(\T) $ to be solutions
	 	of scaled gILW \eqref{ILW3} 
	 	with  initial data
	 	$v_0^{(j)} \in  H^s(\T) $.
	 	Then,
	 	the following holds
	 	\begin{align*}
	 	\|v^{(1)}  -  v^{(2) }\|_{ N^{s-1, \dl}_{T} }  &\les 
	 	\|v^{(1)}   -  v^{(2)}   \|_{ L^{\infty}_{T} H_{x}^{s-1} } \\
	 	&\quad+  C\big( \|   v^{(1)}    \|_{ L^{\infty}_{T} H_{x}^{s} }  +    \| v^{(2)}    \|_{ L^{\infty}_{T} H_{x}^{s} } \big) 
	 	\| v^{(1)}-  v^{(2)}    \|_{L^{\infty}_{T} H_{x}^{s-1} }.
	 	\end{align*}

	 \end{itemize}

	\noi
	Here, the implicit constants are independent of $\dl$.
	
\end{lemma}

\begin{proof}

The proof is based on \cite[Lemma 3.1]{MV15} and \cite[Lemma 4.7]{MT}. 
We notice that the proof itself is independent of the depth parameter $\dl$ and only requires a standard $X^{s,b}$-type analysis.
Therefore, for the sake of conciseness, the details of the proof have been omitted.

%

\end{proof}

In what follows, we will establish our main uniform energy estimates. This argument is inspired by the improved energy method developed by Molinet-Vento \cite{MV15} and can also be found in the work by Molinet-Tanaka \cite[Proposition 4.8]{MT}. 
The idea behind this approach is rooted in the classical energy method and is well suited for our model. The dispersion term vanishes due to integration by parts, and as a result, we obtain the following.
Let $u \in C(\R; H^\infty(\T))$ be a smooth solution of gILW  \eqref{gILW1}.
Then, by the Fundamental Theorem of Calculus, we have
\begin{equation}
\|  u(t)  \|_{L^2}^2 - \|  u (0)  \|^2_{L^2}
=-2 \int_{\T}  \dx (  u^k  )  u  \, dx.
\label{FTC}
\end{equation}

\noi
The challenge then lies in studying the nonlinear interactions on the right-hand side of \eqref{FTC}. 
We note that equation \eqref{FTC} holds exactly for smooth solutions $v \in C(\mathbb{R}; H^\infty(\mathbb{T}))$ of the scaled gILW equation \eqref{ILW3}. Additionally, this type of argument works well on $\mathbb{R}$.

\begin{proposition}[Uniform energy estimate]
	\label{PRO:EN1}

Let $k \geq 2$.
Consider $u_0, v_0 \in H^s(\mathbb{T})$ for $s \geq \frac{3}{4}$ and 
$u,v \in L^\infty([0,T]; H^s(\mathbb{T}))$
 to be solutions of  gILW  \eqref{gILW1} and  scaled gILW  \eqref{ILW3}, respectively, on $[0,T]$ for $0<T<1$.
%
Then,  for
$\{\omega_N\}$ be a dyadic sequence that satisfies \eqref{eq_dydic} with $\kk\geq 1$,
the following statements hold.

	\begin{itemize}
	\item[(i)]	
Let $2\leq \dl\leq \infty$. Then, we have 
	\begin{equation*}
	\|u\|_{L_T^\infty H_\om^s}^2
	\leq \|u_0\|_{H_\om^s}^2 + T^{\frac14} C(\|u\|_{M_{T}^{\frac34,\dl}})
	\|u\|_{M_{\om,T}^{s,\dl}}
	\|u\|_{L_T^\infty H_\om^s}.
	\end{equation*}
	
	\smallskip
	
	\item[(ii)]
Let $ 0<\dl< 1$. Then,
	we have
	\begin{equation*}
	\|v\|_{L_T^\infty H_\om^s}^2
	\le \|v_0\|_{H_\om^s}^2 + T^{\frac14} C(\|v\|_{N_{T}^{\frac34,\dl}})
	\|v\|_{N_{\om,T}^{s,\dl}}
	\|v\|_{L_T^\infty H_\om^s}.
	\end{equation*}
	\end{itemize}

When we only consider  $k=2$,
the statements {\rm (i)} and {\rm (ii)} hold true for $s > \frac 12$ and $\om\equiv 1$.
	
\end{proposition}

\begin{proof}

The proof   is based on \cite[Proposition 4.8]{MT} for general $k\geq 2$ and regularity is needed for $s\geq \frac34$.
When we consider only $k=2$, the proof is based  on  \cite[Proposition 3.4]{MV15}. 
Without loss of generality, we provide a succinct outline of the case involving the general nonlinearity $\dx (u^k)$ for $k\geq 2$.
Throughout the following discussion, we will identify the instances where we will need to replace our previously obtained uniform estimates.

 Taking the $L^2$-scalar product of the resulting equation with $P_N u$, 
	multiplying by $ \omega_N^2 \langle N\rangle^{2s} $ and integrating
	over $[0,t]$ with $0<t<T$,
	we yield
	\begin{equation*}
	\om_N^2 \jb{N}^{2s} \| P_N u(t)\|_{L^2}^2
	=   \om_N^2 \jb{N}^{2s} \|u_0\|_{L^2}^2
	-2    \om_N^2 \jb{N}^{2s}
	\int_0^t\int_\T \dx P_N(u^k) P_N  u \, dxdt'.
	\end{equation*}

	\noi
	We use integration by parts,
	apply Bernstein inequalities, and sum over in $N$,
	we obtain
	\begin{align*}
	\begin{aligned}
	\|u(t)\|_{H_\om^s}^2
	&  =\sum_N \omega_N^2 ( 1 \vee N)^{2s}  
	\bigg(
	\| P_N u_0\|^2_{L^2_x} - 2\int^t_0 \int_\T
	P_N \dx(u^k ) P_N u \,  dx dt'
	\bigg)\\
	&\le \|u_0\|_{H_\om^s}^2
	+2\sum_{N\ge 1}  \om_N^2 N^{2s} \bigg|\int_0^t\int_\T u^k P_N^2 \dx u \,  dxdt'\bigg|\\
	&\leq   \|u_0\|_{H_\om^s}^2+	2I^t_{k},
	\end{aligned}
	\end{align*}

	\noi
	where $I_{k}^t$ is defined by
	\begin{align*}
	I^t_{k} := \sum_{N\ge 1} \om_N^2 N^{2s} \bigg|\int_0^t\int_\T u^k P_N^2\dx u \, dxdt'\bigg|.	
	\end{align*}
 Therefore, 
	we shall prove that for any $ k\ge 1$, the following holds
	\begin{equation}
	\label{eq:emain}
	I_{k+1}^t \leq T^{\frac14} C^k
	(\|u\|_{X_{\om,T}^{s-1,1,\dl}}+\|u\|_{L_T^\infty H_\om^s})
	\|u\|_{L_T^\infty H_\om^s},
	\end{equation}
	
\noi
where $C$  depends only on $\|u\|_{\Mssss}$.
One can easily check that \eqref{eq:emain} holds when we sum over  $N\les 1 $. 
See for example \cite[eq.(4.14)]{MT}.
Therefore, it is enough to consider \eqref{eq:emain} with  $N \gg 1  $.
	First, we define the following symbols:
	\begin{align*} 
	A(n_1,\dots,n_{k+2})
	&:=\sum_{j=1}^{k+2}\phi_N^2(n_j)n_j,\\
	A_1(n_1,n_2)
	&:=\phi_N^2(n_1)n_1+\phi_N^2(n_2)n_2,\\
	A_2(n_4,\dots,n_{k+2})
	&:=\sum_{j=4}^{k+2}\phi_N^2(n_j)n_j.
	\end{align*}

	\noi
	Here, $\phi_N$ is defined in Section \ref{SEC:nota}.
	It is clear that 
	\[
	A(n_1,\dots,n_{k+2})=
	A_1(n_1,n_2)+
	\phi_N^2(n_3)n_3+
	A_2(n_4,\dots,n_{k+2}).
	\]

	\noi
	Moreover,
	we see from the symmetry that
	\begin{equation*}
	\begin{aligned}
	\int_\T u^{k+1}P_N^2\dx u dx
	&=\frac{i}{k+2}\sum_{n_1+\cdots+n_{k+2}=0}A(n_1,\dots,n_{k+2})\prod_{j=1}^{k+2}\ft{u}(n_j)\\
	&=\frac{i}{k+2}\sum_{N_1,\dots,N_{k+2}}\sum_{n_1+\cdots+n_{k+2}=0}A(n_1,\dots,n_{k+2})\prod_{j=1}^{k+2}\phi_{N_j}(n_j)\ft{u}(n_j).
	\end{aligned}
	\end{equation*}

	\noi
	By symmetry, we can assume  that 
	\begin{align*}
	\begin{cases}
	N_1\ge  N_2 \ge N_3, \, \textup{ if } k=1;\\
	N_1\ge N_2 \ge N_3\ge N_4, \,  \textup{ if } k=2 ; \\
	N_1\ge N_2\ge N_3\ge N_4\ge N_5= 
	\displaystyle \max_{j\ge5}N_j, \, \textup{ if } k\ge 3.
	\end{cases}
	\end{align*}
	We notice that
	the cost of this choice is a constant factor less than $(k+2)^4 $.
We also observe that frequency projection  $ P_N$
	ensures that there is no contribution 
	\footnote{
		$P_NP_{N1}$=0 if $N_1\leq \frac N4$, since the $\{ \textup{supp}(P_N)\cap \textup{supp}(P_{N_1})\}=\emptyset$.
	}
for any $ N_1\le  N/4 $. Hence, we can assume that $ N_1 \ge \frac N4 $ and that $ N_2\ges \frac{N_1}{k}  $ with $ N_2\ge 1$.

In the following, we verify the $A_1$ case to illustrate where the uniform estimates play in their roles. Then, the rest of details will follow
 \cite[Proposition 4.8]{MT} by replacing all the relevant estimates.

In $A_1$ contribution,
we observe that the frequency projector in $ A_1 $ ensures that either $ N_1\sim N$ or $ N_2\sim N $,
and in both cases $ N \ges N_3$.  
Moreover, we can further assume that $ N_3\ge 1 $,
otherwise   $ A_1$ contribution will be cancelled by integration by parts.
And then,
we divide $A_1$   contribution into three cases:  
	\begin{align*}
	\textup{(A)}\, N_2\les    N_3\les k N_4 ,
	\quad\quad\quad
	\textup{(B)} \, N_3\gg k N_4  \textup{ or } k=1,
	\quad\quad\quad
	\textup{(C)}\,N_2\gg    N_3 .
	\end{align*}
Let us define the following notation
	\begin{align*}
	J_t
	:=\sum_{N\gg 1}\sum_{N_1,\dots,N_{k+2}}\om_N^2 N^{2s}
	\bigg|\int_0^t\int_\T\Pi(P_{N_1}u,P_{N_2}u)\prod_{j=3}^{k+2}P_{N_j}udxdt'\bigg|,
	\end{align*}

	\noi
	where $\Pi(f,g)$ is defined
	to be
	\begin{equation*}
	\Pi(u,v):=v \dx P_N^2u +u \dx P_N^2v .
	\end{equation*}
		Note that $ N\gg 1 $ ensures that  $N_1\gg1$.

	\smallskip
	\noindent
	\textbf{Case A: $N_2\les  N_3\les k N_4$.}
	\smallskip

In this case, we have $N \les N_1\les k  N_2 \les  k N_3 \les k^2 N_4  $.
And	then, 
the main difference in estimating the following 
from the method presented in \cite{MT} lies in the fact that our uniform linear estimate, as stated in Proposition \ref{PRO:stri}, will result in the last inequality.
Hence,
by combining H\"older's, Bernstein's and Young's inequality, we  show that
	\begin{align*}
	J_t
	&\les_k \sum_{N_1,\dots,N_{k+2}\atop N_1\les k^2 N_4, N_1\ge N_4, N_2\ge N_4,N_3\ge N_4} \om_{N_1}^2 N_4^{2s+1}
	\prod_{j=1}^4\|P_{N_j}u\|_{L_{T,x}^4}
	\prod_{j=5}^{k+2}\|P_{N_j}u\|_{L_{T,x}^\infty}\\
	&\les_k \sum_{N_1\ge N_4,N_2\ge N_4,N_3\ge N_4}
	\frac{\om_{N_1}}{\om_{N_2}} \om_{N_1} \om_{N_2}  \Bigl(\frac{N_4}{N_1}\Bigr)^{s-\frac{1}{8}} \Bigl(\frac{N_4}{N_2}\Bigr)^{s-\frac{1}{8}} \Bigl(\frac{N_4}{N_3}\Bigr)^{\frac{1}{8}+\frac12}\\
	&\quad \times \prod_{j=1}^2
	\|D_x^{s-\frac{1}{8}} P_{N_j}u\|_{L_{T,x}^4}
	\prod_{l=3}^4
	\|D_x^{\frac{1}{8}+\frac12} P_{N_l}u\|_{L_{T,x}^4}\\
	&\les_k \Bigl( \sum_{K}\om_K^4\|D_x^{s-\frac{1}{8}}P_K u\|_{L_{T,x}^4}^4\Bigr)^{\frac12}
	\Bigl( \sum_{K}\|D_x^{\frac{1}{8}+\frac12}P_K u\|_{L_{T,x}^4}^4\Bigr)^{\frac12}\\
	&\les_k  T C \|u\|_{L_T^\infty H_\om^{s}}^2,
	\end{align*}

	\noi
One may notice that Lemma \ref{LEM_ev} ($\kk \le 2$
	and $ N_1\les k N_2 $) implies   $  \frac{\om_{N_1}}{\om_{N_2}} \les k $.
	Moreover,
	it
	is not difficult to see that the last inequality holds when
	$
	s\geq \frac34
	.$

	\smallskip
	\noindent
	\textbf{Case B: $ N_3\gg k  N_4$ or $ k=1$.} 
	\smallskip

For technical reasons we will take the extensions $\wt{u}=\rho_T(u)$ of $u$, which is defined in \cite[Lemma 2.1]{MT}.
Moreover, we define the following functional:
	\begin{equation*}
	J_\infty
	:=\sum_{N\gg 1}\sum_{N_1,\dots,N_{k+2}}\om_N^2 N^{2s}
	\bigg|\int_\R\int_\T\Pi(u_1,u_2)
	\prod_{j=3}^{k+2}u_jdxdt'\bigg|.
	\end{equation*}

	\noi
	By setting $R=N_1^{\frac13}N_3^{\frac43}$, and then we split $J_t$ into
	\begin{align}
	\label{eq:split}
	\begin{aligned}
	J_t
	&\le J_\infty  (P_{N_1}\1_{t,R}^{\textrm{high}}\wt{u},
	P_{N_2}\1_t\wt{u},P_{N_3}\wt{u},\cdots, P_{N_{k+2}}\wt{u})\\
	&\quad+J_\infty   (P_{N_1}\1_{t,R}^{\textrm{low}}\wt{u},
	P_{N_2}\1_{t,R}^{\textrm{high}}\wt{u},
	P_{N_3}\wt{u},\cdots,
	P_{N_{k+2}}\wt{u})\\
	&\quad+J_\infty (P_{N_1}\1_{t,R}^{\textrm{low}}\wt{u},
	P_{N_2}\1_{t,R}^{\textrm{low}}\wt{u},
	P_{N_3}\wt{u},\cdots,
	P_{N_{k+2}}\wt{u})\\
	& =:J_{\infty,1} +J_{\infty,2} +J_{\infty,3}.
\end{aligned}
\end{align}

 We start with $J_{\infty,1}$, and
	recall that $N\sim N_1 \sim N_2$.  
From  \cite[Lemma 3.6]{MPV19}, we have
	\begin{align*}
	\|\1_{t,R}^{{\rm
			high}}\|_{L^1}\les T^{\frac14}N_1^{-\frac{1}{4} }N_3^{-1},
	\end{align*}
	which implies
		\begin{align}
	\begin{aligned}
	J_{\infty,1}
	&\les\sum_{N_1,\dots,N_{k+2}} \om_{N_1}^2
	N_1^{2s}N_3\|\1_{t,R}^{\textrm{high}}\|_{L_t^1}
	\|P_{N_1}\wt{u}\|_{L_t^\infty L_x^2}\|P_{N_2}\wt{u}\|_{L_t^\infty L_x^2}
	\prod_{j=3}^{k+2}\|P_{N_j}\wt{u}\|_{L_{t,x}^\infty}\\
	&\les
	T^{\frac14}\|\wt{u}\|_{L_t^\infty H_x^{s'}}^k\|\wt{u}\|_{L_t^\infty H_\om^{s}}^2
	\sum_{N_1} N_1^{-\frac14}
	\les T^{\frac14}C^k \|u\|_{L_T^\infty H_\om^{s}}^2,
	\end{aligned}
	\label{J21}
	\end{align}

	\noi
	for some $s'>\frac12$.
	By the strategy as for \eqref{J21},  We can estimate 
	$J_{\infty,2}^{(2)}$.

To estimate term  $J_{\infty,3} $.
It is imperative to carry out a further decomposition based on the modulation functions. 
In this regard,  the  crucial difference from \cite{MT} is that we need Lemma \ref{lem_res1} such that the following
uniform lower bound holds:
	\begin{align*}
	|\Od_{k+2} (\wt n)| \ges N_3N_1 \gg R.
		\end{align*}

	\noi
	Then, by defining $L:=N_3 N_1 $, 
	we further decompose $J_{\infty,3} $ and arrive the following
	\begin{align}
	\label{eq:split3}
	\begin{aligned}
	J_{\infty,3} 
	&\le J_\infty   (P_{N_1}Q_{\ges L}(\1_{t,R}^{\textrm{low}}\wt{u}),
	P_{N_2}\1_{t,R}^{\textrm{low}}\wt{u},
	P_{N_3}\wt{u},\cdots,
	P_{N_{k+2}}\wt{u})\\
	&\quad +J_\infty (P_{N_1}Q_{\ll L}(\1_{t,R}^{\textrm{low}}\wt{u}),
	P_{N_2}Q_{\ges L}(\1_{t,R}^{\textrm{low}}\wt{u}),
	P_{N_3}\wt{u},\cdots,
	P_{N_{k+2}}\wt{u})\\
	&\quad +J_\infty  (P_{N_1}Q_{\ll L}(\1_{t,R}^{\textrm{low}}\wt{u}),
	P_{N_2}Q_{\ll L}(\1_{t,R}^{\textrm{low}}\wt{u}),
	P_{N_3}Q_{\ges L}\wt{u},\cdots,
	P_{N_{k+2}}\wt{u})\\
	&\quad+\cdots
	+J_\infty (P_{N_1}Q_{\ll L}(\1_{t,R}^{\textrm{low}}\wt{u}),
	P_{N_2}Q_{\ll L}(\1_{t,R}^{\textrm{low}}\wt{u}),
	P_{N_3}Q_{\ll L}\wt{u},\cdots,
	P_{N_{k+2}}Q_{\ges L}\wt{u})\\
	&=:J_{\infty,3}^{\rm (1)}   +\cdots+J_{\infty,3}^{\rm (k+2)}  .
		 \end{aligned}
	\end{align}

\noi
Hence,	it suffices to estimate each $J_{\infty,3}^{\rm (j)} $ for $j=1,...,k+2$.
In order to control $J_{\infty,3}^{(1)}$, we first observe
 \cite[Lemma 3.6]{MPV19} implies
	$\|  \1_{t,R}^{\textrm{high}}\|_{L^2_t}\leq R^{-\frac12}  $,  and also we have that
	\begin{align}
	\label{eqlow0}
	\begin{aligned}
	\|P_{N_2}\1_{t,R}^{\textrm{low}}\wt{u}\|_{L_{t,x}^2}
	&\le\|P_{N_2}\1_{t}\wt{u}\|_{L_{t,x}^2}
	+\|P_{N_2}\1_{t,R}^{\textrm{high}}\wt{u}\|_{L_{t,x}^2}\\
	&\les \|P_{N_2}\1_{t}\wt{u}\|_{L_{t,x}^2}
	+T^{\frac14}R^{-\frac14}\|P_{N_2}\wt{u}\|_{L_{t}^\infty L_x^2}.
	\end{aligned}
	\end{align}

\noi
Thus,
by using  
\cite[Lemma 2.4]{MT},
 \cite[Lemma 3.7]{MPV19}, H\"older's inequality,  and \eqref{eqlow0}, we obtain
 the following
	\begin{align}
		\label{eq:j1}
		\begin{aligned}
	J_{\infty,3}^{\rm (1)}
	&\les \sum_{N_1,\dots,N_{k+2}}\om_{N_1}^2 N_1^{2s}N_3
	\|P_{N_1}Q_{\ges L} (\1_{t,R}^{{  \rm low}}\wt{u})\|_{L_{t,x}^2}
	\|P_{N_2}\1_{t,R}^{{ \rm low}}\wt{u}\|_{L_{t,x}^2}
	\prod_{j=3}^{k+2}\|P_{N_j}\wt{u}\|_{L_{t,x}^\infty}\\
	&\les_k \|\wt{u}\|_{L_t^\infty H_x^{s'}}^k
	\sum_{N_1 \ges 1}\om_{N_1}^2 N_1^{2s-1}
	\|P_{N_1} \wt{u}\|_{X^{0,1,\dl}}
	\|P_{N_1} \1_t \wt{u}\|_{L_{t,x}^2}\\
	&\quad+T^{\frac14}\|\wt{u}\|_{L_t^\infty H_x^{s'}}^{k-1}
	\sum_{N_1   \ges  N_3}\om_{N_1}^2 N_1^{2s- \frac{13}{12}  }N_3^{- \frac 13}\|P_{N_1} \wt{u}\|_{X^{0,1,\dl}}
	\|P_{N_1} \wt{u}\|_{L_{t}^\infty L_x^2}
	\|P_{N_3}\wt{u}\|_{L_{t,x}^\infty}\\
	&\les T^{\frac14}\|\wt{u}\|_{L_t^\infty H_x^{s'}}^k
	\|\wt{u}\|_{L_t^\infty H_\om^s}\|\wt{u}\|_{X_\om^{s-1,1,\dl}}
	\les_k T^{\frac14}
	\|u\|_{L_T^\infty H_\om^s}\|u\|_{M_{\om,T}^{s,\dl}}.
	\end{aligned}
	\end{align}

	\noi
Moreover,we can immediately estimate $J_{\infty,3}^{(2)}$ by the same approach as \eqref{eq:j1}.

	Next, we consider the contribution $J_{\infty,3}^{(3)}$.
	By using \cite[Lemma 3.5]{MPV19}, \cite[Lemma 2.4]{MT}, H\"older's inequality  yield
	\begin{align}
	\label{eq:j3}
	\begin{aligned}
      J_{\infty,3}^{(3)}
	&\les\sum_{N_1,\dots,N_{k+2}}\om_{N_1}^2  N_1^{2s}N_3
	\|P_{N_1}Q_{\ll L}
	(\1_{t,R}^{ {\rm low}} \wt{u})\|_{L_{t,x}^2}
	\|P_{N_2}Q_{\ll L}(\1_{t,R}^{ {\rm low}}\wt{u})\|_{L_{t}^\infty L_x^2}\\
	&\quad\times\|P_{N_3}Q_{\ges L}
	\wt{u}\|_{L_{t}^2 L_x^\infty}
	\prod_{j=4}^{k+2} \|P_{N_j}\wt{u}\|_{L_{t,x}^\infty}\\
	&\les_k T^{\frac12}  \| \wt{u} \|_{L_{t}^\infty H_x^{s'}}^{k-1}
	\sum_{N_1\ges N_3\geq 1} \om_{N_1}^2
	N_1^{2s-1} \|P_{N_1} \wt{u}\|_{L_{t}^\infty L_x^2}^2
	\|D_x^{\frac12} P_{N_3}\wt{u}\|_{X^{0,1,\dl}}\\
	&\les T^{\frac12}\|\wt{u}\|_{L_{t}^\infty H_x^{s'}}^{k-1}
	\sum_{N_1\ges N_3\geq 1}
	N_1^{-\frac{1}{8}}N_3^{-\frac{1}{8}}
	\om_{N_1}^2 N_1^{2s}\|P_{N_1}\wt{u}\|_{L_{t}^\infty L_x^2}^2
	\|P_{N_3}\wt{u}\|_{X^{-\frac14,1,\dl}}\\
	&\les T^{\frac12}\|\wt{u}\|_{L_{t}^\infty H_x^{s'}}^{k-1}
	\|\wt{u}\|_{X^{-\frac14,1,\dl}}\|\wt{u}\|_{L_t^\infty H_\om^s}^2
	\les   T^{\frac12}\|u\|_{L_T^\infty H_\om^s}^2.
	\end{aligned}
	\end{align}

	\noi
Moreover,
by a similar argument as in $J_{\infty,3}^{(3)} $\eqref{eq:j3},
	we get 
$
	J_{\infty,3}^{(j)}    \les T^{\frac12}\|u\|_{L_T^\infty H_\om^s}^2,
$
for all $4\leq j\leq k+2$.

	\smallskip
	\noi
	\textbf{Case C: $N_1\sim N_2\gg  N_3$.}
	\smallskip

	In this case, we need to compare the size $|n_3+n_4| $ and $ k |n_5|$.
	By symmetry we can assume $|n_5|\geq |n_j|$,   where $ n_j $ is the $j$-th largest frequency. 
	Therefore, we consider the following two cases:
	\[
	|n_3+n_4|\gg k|n_5| 
	\quad\quad\quad \quad\quad\quad
	|n_3+n_4|\les k|n_5|.
	\]

	\noi
	If $|n_3+n_4| \gg k |n_5| $, we have a good non-resonance interaction (see Lemma \ref{lem_res2}). Then, we can finish the proof in a similarly way of Case B.
	Otherwise, we are in the ``almost" resonance situation.
	In particular, we can share the lost derivative between three functions, $P_{N_j} u$, for $j=3,4,5$, and then, we finish the proof by using
our uniform estimate  Proposition \ref{PRO:stri}, which is similar to Case A.

\end{proof}

\begin{remark}
\rm

The above proof showed two crucial differences in our scenario, namely the uniform linear estimates provided by Proposition \ref{PRO:stri} 
and the uniform lower bounds on the resonance functions given in Lemmas \ref{lem_res1} and \ref{lem_res2}. 
When $k=2$, Case B is coincide with the argument in \cite{MV15} for $s>\frac 12$.
 In particular, we see in the proof of \cite[Lemma 3.2]{MV15},
the first decomposition corresponding to \eqref{eq:split} is  \cite[eq.\,(3-5)]{MV15} and then the second decomposition 
corresponding to \eqref{eq:split3} is \cite[eq.\,(3-7)]{MV15}. By using Lemmas \ref{lem_res1} we can obtain the same estimate.

\end{remark}

\subsection{Uniform difference estimates}

In what follows, we will establish the difference estimate at the regularity level $s-1$. This is necessary because the symmetrization argument that we used in the proof of Proposition \ref{PRO:EN1} is less effective when dealing with the difference between two solutions. Similar arguments can be found in \cite{MV15, MT}.
Let us consider two solutions $(u^{(1)}, u^{(2)}) \in (M_T^{s,\dl})^2$ of the gILW equation \eqref{gILW1} with initial data $(u_0^{(1)}, u_0^{(2)}) \in (H^s(\T))^2$. We denote the difference between these solutions as $w = u^{(1)} - u^{(2)}$, which satisfies the equation
\begin{equation}
\label{eq_w}
\dt w - \mathcal{G}_{\dl} \dx^2 w = \dx( (u^{(1)})^k - (u^{(2)})^k  ).
\end{equation}
The goal is to establish analogous estimates to those in Proposition \ref{PRO:EN1} for the equation \eqref{eq_w}.

\begin{proposition}
	\label{PRO:EN2}

Let $k \geq 2$. For $j = 1,2$, 
let $u_0^{(j)}, v_0^{(j)} \in H^s(\mathbb{T})$ for $s \geq \frac{3}{4}$ and 
 let $u^{(j)} \in M_T^{s,\delta}$
 to be solutions of  gILW  \eqref{gILW1},
 $v^{(j)} \in N_T^{s,\delta}$  to be solutions of 
  scaled gILW  \eqref{ILW3}, respectively, on $[0,T]$ for $0<T<1$.
 Then,
the following hold.

	\begin{itemize}
		\item[(i)]
Let $2\leq \dl\leq \infty$. Then, we have 
		\begin{align*}
		\begin{aligned}
		\| u^{(1)}- u^{(2)}  \|_{L_T^\infty H_x^{s-1}}^2
		&\le \|u^{(1)}_0 - u^{(2)}_0\|_{H^{s-1}}^2 +T^{\frac14}
		C \big(  \| u^{(1)}  \|_{M^{s,\dl}_T}   ,  \|u^{(2)}   \|_{M^{s,\dl}_T}   \big)\\
		&\quad \times
		\| u^{(1)}- u^{(2)}   \|_{M^{s-1,\dl}_T}
		\|  u^{(1)}- u^{(2)}   \|_{L_T^\infty H_x^{s-1}}.
		\end{aligned}
		\end{align*}

\item[(ii)]
Let $ 0<\dl< 1$. Then,
	we have
		\begin{align*}
		\begin{aligned}
		\| v^{(1)}- v^{(2)}    \|_{L_T^\infty H_x^{s-1}}^2
		&\le \|v^{(1)}_0 - v^{(2)}_0\|_{H^{s-1}}^2 +T^{\frac14}
		C \big(  \| v^{(1)}  \|_{N^{s,\dl}_T}     ,  \|     v^{(2)}   \|_{N^{s,\dl}_T}  \big)\\
		&\quad \times
		\|   v^{(1)}- v^{(2)}    \|_{N^{s-1,\dl}_T}
		\|   v^{(1)}- v^{(2)}   \|_{L_T^\infty H_x^{s-1}}.
		\end{aligned}
		\end{align*}

	\end{itemize}

When we only consider  $k=2$,
the statements {\rm (i)} and {\rm (ii)} hold true for $s > \frac 12$.

\end{proposition}

\begin{proof}

For simplicity, let us denote the two solutions are $u,v \in M_T^{s,\dl}$
associated with the initial data $u_0,v_0 \in H^s(\T) $.
The difference  $w=u-v$ satisfies
\eqref{eq_w}.
Moreover, we have the following
	\begin{equation}
	\label{df1}
u^k-v^k=\sum_{k\ge 2} 
	\sum_{i=0}^{k-1}  w u^i  v^{k-1-i}.
	\end{equation}

	\noi
	We
	proceed as in the proof of Proposition \ref{PRO:EN1} to  see from \eqref{eq_w} that for $t\in[0,T]$
	and we obtain
	\begin{align*}
	\|w(t)\|_{H_x^{s-1}}^2
	\le \|u_0-v_0\|_{H^{s-1}}^2
	+  2\sum_{k\ge 2}   	\max_{i\in \{0,..,k-1\}} I_{k,i}^t ,
	\end{align*}

	\noi
	where $I_{k,i}^t $ is defined as
	\begin{align*}
	I_{k,i}^t:= \sum_{N\ge 1}  N^{2(s-1)} \bigg|\int_0^t\int_\T  u^i  v^{k-1-i} wP_N^2\dx w\,  dxdt'\bigg| .
	\end{align*}

	\noi
	Therefore we are reduced to estimating the contribution of
	\begin{equation*}
	I_{k+1}^t=  \sum_{N\ge 1}  N^{2(s-1)} \bigg|\int_0^t\int_\T  \textbf{z}^k wP_N^2\dx w\,  dxdt'\bigg| \;
	\end{equation*}

	\noi
	where
	now we take $k\geq1$ and
	 $\textbf{z}^k $ stands for $ u^i v^{k-i} $ for some $i\in \{0,..,k\}$.  
	We set
	$
	C:=C(\|u\|_{M^{s,\dl}_T}+\|v\|_{M^{s,\dl}_T})
	$,
	%
	%
	%
	%
	%
	and we suffice to show for any $ k\ge 1$ the following bound holds
	\begin{equation}
	\label{eq:dfmain1}
	I_{k+1}^t \le T^{\frac14}   C^k 
	\|w\|_{M^{s-1,\dl}_T}
	\|w\|_{L_T^\infty H^{s-1}_x}.
	\end{equation}
	
\noi
The proof of \eqref{eq:dfmain1} follows similar to \cite[Proposition 5.1]{MT}  (\cite[Proposition 3.5]{MT}, when we only consider $k=2$),
in the way we need to use the uniform estimates from Proposition \ref{PRO:stri} and Corollary \ref{COR:stri} and the uniform lower bounds on the resonance functions from Lemmas \ref{lem_res1} and \ref{lem_res2}.
This idea we have already saw in the proof of 	Proposition \ref{PRO:EN2},
for the sake of conciseness, the details of the proof have been omitted here.

\end{proof}

\section{Convergence of the ILW-type equations}
\label{SEC:limit1}

We have previously discussed the two potentially singular limits of the ILW-type equation in Subsection \ref{SUB:LE}, namely $\dl \to \infty \backslash 0$.
  In the following analysis, we emphasize the dependence of $\dl$, $u = u_\dl$, and write the equation in the form
\begin{equation}
\label{gILWd}
\dt u_\dl -
\G \partial_x^2 u_\dl  =  \dx ( u_\dl)^k
\end{equation}

\noi 
to reflect this dependence.
 In this section, we aim to   construct a perturbative argument to show the convergence of solutions of ILW-type. 
  Our first goal is to establish that the solutions remain uniformly bounded with respect to the depth parameter $\delta$.
   This result ensures that the solutions remain well-defined as $\delta$ approaches zero or infinity.
  Subsequently, 
  we show on the limits of $\delta$ as it approaches  infinity (deep-water limit) and zero (shallow-water limit).

It should be noted that the methodology developed in this section can be easily applied to the ungauged methods outlined in \cite{MV15, MT} to study the convergence of ILW-type equations.
For the purpose of brevity, we will only present the full details of the convergence of gILW \eqref{gILW1} and scaled gILW \eqref{ILW3}. 
Nonetheless, the proof of  the specific $k=2$ and $s>\frac12$ can be achieved by using a similar methodology. We will explain where the changes are along our proof in the following discussion.

\subsection{Uniform control on the solutions}
\label{SEC:Lo}

In this subsection, we aim to show that the solutions of the equations are uniformly controlled with respect to the depth parameter $\delta$. 
One crucial observation is that the results of \cite[Theorem 1.1]{MT} in conjunction with our uniform estimates Propositions \ref{PRO:EN1} and \ref{PRO:EN2} imply the following uniform well-posedness results.

\begin{proposition}[Uniform well-posedness]
	\label{PRO:LWP2}

Let  $k\geq 2$ and $s\geq \frac34$. Then,
the   gILW  equation \eqref{gILW1} and   the scaled gILW  equation  \eqref{ILW3} are  unconditionally locally well-posed in $H^s(\T)$. 
The maximal time of existence  for both  gILW   and  scaled gILW, is dependent only on the initial data. 
 Moreover,   for any	$0<\dl\leq \infty$, 
  the following estimates hold:
  \begin{align}
 \|u_{\dl}\|_{C_{T} H^{s}(\T) }  \leq 2\|u_{0}\|_{H^{s}(\T)}
   \quad\quad \text{and} \quad\quad
 \|v_{\dl}\|_{C_{T} H^{s}(\T) } \leq 2\|v_{0}\|_{H^{s}(\T)}.
  \label{eq:LWP2}
 \end{align}

 Furthermore,
when we only consider  $k=2$,
the statements hold true for $s > \frac 12$.

\end{proposition}

\begin{proof}

This proof  follows from \cite[Secrion 6]{MT}, in the way we need to apply our uniform estimates Propositions \ref{PRO:EN1} and \ref{PRO:EN2}.
When $k=2$ only, see in \cite[Section 3B]{MV15}.
\end{proof}

\begin{remark}
\rm
\label{RM:cts}

The  frequency parameter $\omega_N$ we saw in the case of treating general nonlinearity is used to show
continuity of the flow map. 
Moreover, by Lemma \ref{LEM_ev} gives that $\omega_N$
 depends only on the initial data and the approximation sequence of the initial condition such that 
\begin{equation*}
\label{env}
\|u_0\|_{H^s_\om} <\infty , 
\quad \quad
\sup_{n \ge 1} \|u_{0,n} \|_{H^s_\om} <\infty,
\quad \quad 
\text{where} \quad  \om_N  \too\infty \; .
\end{equation*}

\noi
By applying
Proposition \ref{PRO:EN1}, we have
the following 
\begin{align*}
\begin{aligned}
\|P_{\leq K}  u_n -u_n\|_{L^\infty_T H_{x}^s}^2 &=
\sum_{N>K} \|P_{N}  u_n     \|_{L^\infty_T H_{x}^s }^2   \leq
\sup_n
\sum_{N>K}
{\om_N^{-2}  } (
  {\om_N^{2}  }
\|P_{N}  u_{n}     \|_{ L^\infty_T  H_{x}^s }^2
 )\\
&\leq
{\om_K^{-2}  }  \sup_n
\| u_n\|^2_{L^\infty_T  H_\om^s}
\les
{\om_K^{-2}  } 
\sup_n
\| u_{0,n}     \|^2_{ H^s_{\om} }
< {\eps}.
\end{aligned}
\end{align*}

\noi 
 see the   relevant discussion  on continuity with respect to initial data in \cite{KT, MT}. 
 This delicate argument enables us to show that the high-frequency component of the smooth approximating solution, say $u_n$\footnote{Here, the sequence $u_{n}$ represents solutions to our gILW equation \eqref{gILW1} generated from the initial data $u_{0,n}$ which converges to $u_{0}$ in the $H^{s}(\M)$ norm.}, can be made arbitrarily small.
\end{remark}

\subsection{The deep-water limit}
\label{SUB:limit1}

In this subsection, we will complete the proof of Theorem \ref{THM:1}. 
We will treat the equation \eqref{gILWd} as a perturbation of the gBO equation, similar to how we treated \eqref{PE1}.
\begin{equation}
\label{eq:p1}
\partial_t u_\dl - \H  (\partial_x^2 u_\dl) +
\cK_\dl \dx u_\dl
=  \dx (u_\dl^k),
\end{equation}

\noi 
where recall    $\cK_\dl = (\H-  \G )\dx $
such that 
\begin{align}
\label{defq1}
\ft{\cK_\dl u}(n) = q_{\dl}(n) \ft u (n)
\end{align}

\noi
and $q_{\dl}(n)$  is defined to be the Fourier symbol
\begin{align}
\label{defq2}
q_{\dl}(n) 
=  \frac{1}{  \dl } -  n  \coth(\dl n)+ |n| .
\end{align}

\noi
It is clear that Lemma \ref{LEM:p1d} and \eqref{defq2} imply that  
for all $n\in \ZP$,
\begin{align}
\label{defq3}
0\leq q_{\dl}(n) \leq 
\frac{2}{\dl}.
\end{align}

\noi
From \eqref{defq3}, it is evident that the operator $\cK$ is of order 0 and it is bounded on all Sobolev spaces $H^s(\T)$. Additionally, by checking \eqref{eq:p1} and \eqref{defq3}, it is clear that as $\delta \to \infty$, we obtain the gBO equation formally. To rigorously jusstify that $u_\delta$ is a solution of the gBO equation as $\delta \to \infty$, we will prove that $\{u_\delta\}_{\delta \geq 1}$ forms a Cauchy sequence in a suitable space.

For any $2 \leq \delta, \gamma \leq \infty$, let $u_\gamma$ and $u_{\delta}$ be solutions of $\textrm{gILW}_{\gamma}$ and $\textrm{gILW}_{\delta}$ (\eqref{gILWd} with different depth parameters), respectively. We define the difference between them as $w = u_{\gamma} - u_{\delta}$, which solves the following initial value problem:
\begin{equation}
\label{ILW5}
\begin{cases}
\dt w -\mathcal H(\dx^2 w)+ \cK_\dl  (\dx w) 
= (\cK_\gamma-  \cK_\dl)\dx u_{\g}-
\dx ( u_{\g}^k - u_{\dl}^k  ) \\
w(x,0) = 0.
\end{cases}
\end{equation}

\noi
Moreover, it is advantageous to use the following notation:
\begin{align}
\label{dfT}
T_{\dl,\gamma}(u)=\cK_\gamma(u) -  \cK_\dl( u).
\end{align}

\begin{proposition}
\label{LM:ca1}
Let $k\geq 2$, $s\geq \frac34$ and $0<T<1$.
Then, the one-parameter family of solutions $\{u_\dl\}_{\dl\geq 2} $
is Cauchy in $C([0,T];H^s(\T))$ as $\dl\to \infty$.
Moreover, when $k = 2$ only the statement holds for $s > \frac 12$.
\end{proposition}

Before proving Proposition \ref{LM:ca1}, we will require the following lemma, which investigates the properties of $M^{s,\delta}_{T}$ for different values of $\delta$.

\begin{lemma}
	\label{LEM:ca1b}
	Let $k\geq 2$ and 
	$0<T<1$.
	Assume that	$s>\frac12$ and $u_{\g}\in L^\infty([0,T];H^s(\T))$ be the solution of $\textrm{gILW}_{\g}$ \eqref{gILWd}
	associated with initial data $u_0 \in H^s(\T)$.
	Then, 
	for any $  2\leq \dl, \g\leq \infty $,
	we have
	\begin{align}
	\label{eq:eq1}
	\| u_{\g} \|_{M_{T}^{s,\dl}}
	\leq C \|u_{\g} \|_{M^{s,\g}_{T}},
	\end{align}

	\noi
	Moreover, there exits a universal constant $C>0$ such that the following holds:
	\begin{align}
	\label{eq:eq2}
	\| u_{\g} \|_{M_{T}^{s,\dl}}
	\leq C \|u_{\g}\|_{L^{\infty}_{T} H^{s}_{x}  } + C(\|u_{\g}\|_{L^{\infty}_{T,x} }   )
	\|u_{\g} \|_{L^{\infty}_{T} H^{s}_{x} }
	\end{align}

\end{lemma}

\begin{proof}

Firstly, Lemma \ref{lem1} implies
$u_{\g}\in M^{s,\g}_{T}$, where 
we recall the definition of $M^{s,\g}_{T}=L^{\infty}_{T}H^{s}_{x}\cap X^{s-1,1,\g}_{T}$.
	Moreover, 
	we recall the symbol   $\pdeep_{\dl}(n)$ from \eqref{dispersionP}, and notice we can write the following 
	\begin{align*}
	\tau - \pdeep_{\dl}(n) &=\tau - \pdeep_{\g}(n) - (\pdeep_{\dl} - \pdeep_{\g}(n))\\
	&=\tau - \pdeep_{\g}(n)  +n(q_{\dl}(n) -q_{\g}(n) ),
	\end{align*}

	\noi
	where $q_{\dl}(n)$ is defined in \eqref{defq2}.
	Therefore, by the triangle inequality and \eqref{defq3} we obtain 
	\begin{align}
	\label{uu1}
	\jb{ \tau - \pdeep_{\dl}(n) } \les 
	\jb{\tau - \pdeep_{\g}(n)}  +   (\dl^{-1}+\g^{-1})  \jb{n}.
	\end{align}

\noi 
	Now, by the definition of $X^{s,b,\dl}$-space \eqref{deX}  and \eqref{uu1},
		for any $ \g,\dl\geq 2$ we have
	\begin{align*}
	\| u_{\g} \|_{X^{s-1,1,\dl}_{T}  }
	&\les \| u_{\g} \|_{X^{s-1,1,\g}_{T}  }
	+(\dl^{-1}+\g^{-1})   \| u_{\g} \|_{X^{s,0}_{T}  }\\
	&= \| u_{\g} \|_{X^{s-1,1,\g}_{T}  }
	+(\dl^{-1}+\g^{-1})   \| u_{\g} \|_{ L^{2}_{T} H^{s}_{x}  }\\
	&\les  \| u_{\g} \|_{X^{s-1,1,\g}_{T}  }
	+    \| u_{\g} \|_{ L^{\infty}_{T} H^{s}_{x}  }.
	\end{align*}
	
	\noi
	In particular, we see
	\begin{align*}
	\| u_{\g} \|_{M^{s,\dl} _{T}}
	&=	\| u_{\g} \|_{X^{s-1,1,\dl}_{T}  } 
	+ \| u_{\g} \|_{ L^{\infty}_{T} H^{s}_{x}  }\\
&	\les
	 \| u_{\g} \|_{X^{s-1,1,\g}_{T}  }
	 +    \| u_{\g} \|_{ L^{\infty}_{T} H^{s}_{x}  } \les  \|u_{\g}\|_{M^{s,\g}_{T}}
	\end{align*}
	
	\noi
	holds
	for any $ \g,\dl\geq 2$. This shows \eqref{eq:eq1}.

	To obtain a more explicit bound for \eqref{eq:eq2}, we will perform the $X^{s,b}$-analysis similar to that of Lemma \ref{lem1}.
	 For $u_{\gamma}$ satisfies the Duhamel formulation, it suffices to check the following:
	\begin{align*}
	\| u_{\g} \|_{X^{s-1,1,\g}_{T}  }
	\les\| u_{\g} \|_{ L^{\infty}_{T} H^{s}_{x}  }
	+C(\|u_{\g}\|_{L^{\infty}_{T,x} })
	\| u_{\g} \|_{ L^{\infty}_{T} H^{s}_{x}  },
	\end{align*}

	\noi
	which follows directly from \eqref{eq:ms}. Therefore,  for any $\g,\dl\geq2$, we have
	\begin{align*}
	\| u_{\g} \|_{M^{s,\dl}_{T}  }
	\les  \|u_{\g}\|_{M^{s,\g}_{T}}
	\les \| u_{\g} \|_{ L^{\infty}_{T} H^{s}_{x}  } +C(\|u_{\g}\|_{L^{\infty}_{T,x} })
	\| u_{\g} \|_{ L^{\infty}_{T} H^{s}_{x}  }.
	\end{align*}
	
	\noi
This finishes the proof of \eqref{eq:eq2}.

\end{proof}

\begin{lemma}
	\label{LEM:ca1c}
Let $k\geq 2$ and  $0 < T < 1$.
Assume that  $s > \frac{1}{2}$,
 $u_{\delta}, u_{\gamma} \in L^\infty([0,T];H^{s}(\T))$ be solutions of $\textrm{gILW}_{\gamma}$ and $\textrm{gILW}_{\delta}$ \eqref{gILWd} with initial data $u_{0,\gamma}, u_{0,\delta}\in H^s(\T)$, respectively.
	Then,  for any
	$2\leq \g, \dl \leq \infty $ such that $w=u_{\g}-u_{\dl}$,
	we have
	\begin{align}
	\label{eq:eq3}
	\begin{aligned}
	\| w\|_{M_{T}^{s-1,\dl}}
	&\les 
	\| w \|_{ L^{\infty}_{T} H^{s-1}_{x}  } +C(\|u_{\g}\|_{L^{\infty}_{T} H^{s}_{x} }  +\|u_{\dl}\|_{L^{\infty}_{T} H^{s}_{x} } ) \|w \|_{ L^{\infty}_{T} H^{s-1}_{x}    }\\
	& \qquad
	+    (\dl^{-1}+\g^{-1})\| u_{\g}\|_{L^{\infty}H^{s-1}_{x} }.
	\end{aligned}
	\end{align}
	
\end{lemma}

\begin{proof}
	
We begin by taking the difference of $\textrm{gILW}_{\g}$ and $\textrm{gILW}_{\dl}$. Then, $w=u_{\g}-u_{\dl}$ satisfies the following equation 
	\begin{align}
	\label{eqw1}
	\dt w -\G \dx^{2} w =\dx (u_{\g}^k) - \dx (u_{\dl}^k)-( \G - \mathcal G_{\g}) \dx^2 u_{\g}.
	\end{align}
	
\noi
In addition, $w = u_{\gamma} - u_{\delta}$ satisfies the Duhamel formulation of the equation \eqref{eqw1}. Let us recall that
$M^{s-1,\g}_{T}=L^{\infty}_{T}H^{s-1}_{x}\cap X^{s-2,1,\g}_{T}$. 
To show \eqref{eq:eq3}, it suffices to estimate $w(x,t)$ in the   $X^{s-2,1,\dl}$-norm. This argument can be found in \cite[Lemma 4.7]{MT} and \cite[Lemma 3.1]{MV15}. In particular, by utilizing the Duhamel formulation of the equation \eqref{eqw1} and performing a similar $X^{s,b}$-analysis as was used to estimate \eqref{eq:mdf}, along with Corollary \ref{LEM:p2b}, we have:
 	\begin{align*}
	\|w   \|_{X^{s-2,1,\dl}_T}
	& \les   \|u_{0,\g} - u _{0,\dl} \|_{  H^{s-1}_{x}  } +       \|   u_{\g}^k   -      u _{\dl}^k  \|_{L^2_T H_x^{s-1}}
	+\|  ( \G- \mathcal G_{\g}) \dx^{2} u_{\g} \|_{X^{s-2,0,\dl}_T} \\
	&  \les   \|   w \|_{ L^\infty_T  H^{s-1}_{x}  }  +C 
	\|  w  \|_{L^\infty_T H_x^{s-1}}
	+\|  ( \G- \mathcal G_{\g}) \dx  u_{\g} \|_{L^\infty_T H_x^{s-1}}  \\
	&  \les     \|   w \|_{ L^\infty_T  H^{s-1}_{x}  }    +
C \|  w  \|_{L^\infty_T H_x^{s-1}}
	+(\dl^{-1}+\g^{-1})\|  u_{\g} \|_{L^\infty_T H_x^{s-1}},
	\end{align*}

\noi
where the constant $C=C( \|u_{\dl} \|_{L^\infty_T H_x^{s}}     +   \|  u_{\g}  \|_{L^\infty_T H_x^{s}})    $
depends only on $ \|u_{\dl} \|_{L^\infty_T H_x^{s}}  , \|  u_{\g}  \|_{L^\infty_T H_x^{s}} $.

\end{proof}

\begin{proof}[Proof of Proposition \ref{LM:ca1}]

Let us first show that $\{ u_{\dl} \}_{\dl\geq2} \subset C([0,T];H^{s-1}(\T))$ is  a Cauchy sequence
for $s\geq \frac34$.
Namely,
for	 $s\geq \frac34$, 
	 any $2\leq \g,\dl\leq \infty$, and $0<T<1$.
	There exists $C= C( \|u_0\|_{H^s(\T)}) > 0$
	independent of $\dl,\gamma$  such that
	\begin{align}
	\label{pri0}
	\|w(t)\|_{C_{T}H^{s-1}_{x}}
	\leq C  \Big(\frac{1}{\dl}+\frac{1}{\gamma} \Big),
	\end{align}

	\noi
	where $w=u_{\g}-u_{\dl}$.	
To prove this, we rewrite $u^k - v^k$ as in  \eqref{df1}.
Then,
following the steps outlined in Proposition \ref{PRO:EN2}, we arrive at our desired result
	\begin{align}
	\label{L1}
	\begin{aligned}
	\|w(t)\|^2_{H^{s-1}_{x}}
	&\les
	\sum_{N\geq 1} N^{2s-2}
	\| T_{\dl,\gamma} P_{N} \dx u_{\g}) \|_{L^{2}_{T,x}} 
	\| P_{N} w \|_{L^{2}_{T,x}} \\
	&+
	2\sum_{k\ge 2}  	\max_{i\in \{0,..,k-1\}} I_{k,i}^{t,\dl,\g} ,
	\end{aligned}
	\end{align}

	\noi
	where
	$T_{\dl,\gamma}(u) $ is defined in \eqref{dfT} and 
	 $I_{k,i}^{t,\dl,\g} $ is defined to be
	\begin{align}
	\label{energybbd}
	I_{k,i}^{t,\dl,\g}:= \sum_{N\ge 1}  N^{2(s-1)} \bigg|\int_0^t\int_\T  u_{\dl}^i  u_{\g}^{k-1-i} w  P_N^2\dx w  dxdt'\bigg| ,
	\end{align}

\noi
 for some $i\in \{0,..,k\}$.  
We see from equation \eqref{L1} that the analysis now reduces to estimating the linear perturbation and nonlinear interaction. The nonlinear interaction, as defined in \eqref{energybbd}, corresponds to the energy estimate and is essentially equivalent to Proposition \ref{PRO:EN2}.
We shall first address the linear perturbation on the right-hand side of \eqref{L1}. To do so, we apply the equations \eqref{defq1}, \eqref{defq3}, \eqref{dfT}, and Cauchy's inequality to arrive at the following 
	\begin{align*}
	\sum_{N\geq 1} N^{2s-2}
	\| T_{\dl,\gamma} P_{N} \dx u_{\g}) \|_{L^{2}_{T,x}} 
	\| P_{N} w \|_{L^{2}_{T,x}} 
	\les C
	\Big(  \frac{1}{\dl}+\frac{1}{\gamma}  \Big)^2   \|u_{\g}\|^2_{L^{\infty}_{T} H^s_{x}}+
	c\|w\|_{L^{\infty}_{T} H^{s-1}_{x}  }^2,
	\end{align*}

	\noi
	where
the small value of $c$ in the result is a consequence of the application of Cauchy's inequality. 
The nonlinear interaction on the right-hand side of \eqref{L1} is precisely given by Proposition \ref{PRO:EN2}. 
In particular,
see \eqref{eq:dfmain1} for the relevant discussion (and to \cite[(3-25),(3-31)]{MV15} for the corresponding discussion of ILW).
In particular, any $k \geq  1$ we have 
	\begin{equation}
	\label{eq:dfmain}
I_{k,i}^{t,\dl,\g} \le  T^{\frac14} 	C(\|u_{\g}\|_{M^{s,\dl}_T}+\| u_{\dl} \|_{M^{s,\dl}_T})
	\|w\|_{M^{s-1,\dl}_T}
	\|w\|_{L_T^\infty H^{s-1}_x}.
	\end{equation}

\noi
Hence, it can be inferred that \eqref{L1} is bounded as indicated below
	\begin{equation}
	\begin{split}
	\|w(t)\|^{2}_{L^\infty_T H^{s-1}_{x}}
	&\leq C
	\Big(\frac{1}{\dl}+\frac{1}{\gamma}   \Big)^{2}  \|u_{\g}\|^{2}_{L^\infty_T H_x^s}+
	c   \|w(t)\|^{2}_{L^\infty_T H_x^{s-1}}\\
	&\quad+
	T^{\frac14}
	C(\|u_{\g}\|_{M^{s,\dl}_T}+\| u_{\dl} \|_{M^{s,\dl}_T})
	\|w\|_{M^{s-1,\dl}_T}
	\|w\|_{L_T^\infty H_x^{s-1}}.
	\end{split}
	\label{Ldif}
	\end{equation}

	\noi
	Now,
	 Lemmas \ref{lem1} and  \ref{LEM:ca1b} imply that
	for any $2\leq \dl,\g \leq\infty$,
	\begin{align}
	\label{eq:dfa}
	\begin{aligned}
	\|u_{\g}\|_{M^{s,\dl}_T}+\| u_{\dl} \|_{M^{s,\dl}_T}
	&\les
	\|u_{\g}\|_{L^{\infty}_{T} H^{s}_{x}  } +C(\|u_{\g}\|_{L^{\infty}_{T,x} }   )
	\|u_{\g} \|_{L^{\infty}_{T} H^{s}_{x} }\\
	&\quad+ \|u_{\dl}\|_{L^{\infty}_{T} H^{s}_{x}  } +C(\|u_{\dl}\|_{L^{\infty}_{T,x} }   )
	\|u_{\dl} \|_{L^{\infty}_{T} H^{s}_{x} }.
	\end{aligned}
	\end{align}

	\noi
	Next, we also need to estimate $\|w\|_{M^{s-1,\dl}_T}$ in \eqref{Ldif}, which follows from Lemma \ref{LEM:ca1c} and $w(x,0)=0$ that
	\begin{align}
	\label{eq:dfb}
		\begin{aligned}
	\| w\|_{M_{T}^{s-1,\dl}}
	&\les C(\|u_{\g}\|_{L^{\infty}_{T} H^{s}_{x} }  ,\|u_{\g}\|_{L^{\infty}_{T} H^{s}_{x} } ) 
\|w \|_{ L^{\infty}_{T} H^{s-1}_{x}    }\\
&\quad	+ C (\dl^{-1}+\g^{-1})\| u_{\g}\|_{L_{T}^{\infty}H^{s-1}_{x} }
	\end{aligned}
	\end{align}

\noi
Moreover, 
according to Proposition \ref{PRO:LWP2}, for any value of $s \geq \frac{3}{4}$, 
$ \{u_\dl\}_{\dl\geq 1} \subset C([0,T];H^s(\T)) $
 is bounded independently of $\dl$.
 Moreover, time $T$ depends only on the initial data.
To elaborate,
 using equation \eqref{eq:LWP2}, we can deduce that there exists a universal constant $M$ such that the aforementioned set is uniformly bounded
	\begin{equation}
	\label{eq:l1}
	\|u_\dl\|_{C_T H^s_x}\les \|u_0\|_{H^s}< M.
	\end{equation}

	\noi
Consequently, based on equations \eqref{eq:dfa}, \eqref{eq:dfb}, and \eqref{eq:l1}, it can be inferred that there exists a universal constant $C>0$ that satisfies the the following 
 	\begin{align}
		\begin{aligned}
	\label{eq:dfc}
	\| w\|_{M_{T}^{s-1,\dl}}+
		\|u_{\g}\|_{M^{s,\dl}_T}+\| u_{\dl} \|_{M^{s,\dl}_T}
	&\leq
	C+
C\|w \|_{ C_T H^{s-1}_{x}    }\\
	&\quad + C (\dl^{-1}+\g^{-1})\| u_{\g}\|_{C_T   H^{s-1}_{x} }.
	\end{aligned}
	\end{align}

	\noi
Therefore, by combining equations \eqref{eq:dfc}, \eqref{eq:l1}, and applying Cauchy's inequality to \eqref{Ldif}, we obtain the desired result,
	\begin{align*}
	\|w(t)\|^2_{C_T H^{s-1}_x}
	&\leq C
	\Big( \frac{1}{\dl}+ \frac{1}{\gamma} \Big)^2 M^{2}+
	c   \|w(t)\|_{C_T H_x^{s-1}}^2\\
	&\quad+
	C  T^{\frac14}
	\|w(t)\|^2_{C_T  H_x^{s-1}}
	+\frac{T^{\frac12}    M^{2} }{2}  \Big( \frac{1}{\dl}+\frac{1}{\gamma}   \Big)^{2} + \frac12 \|w(t)\|^{2}_{C_T  H_x^{s-1}} \\
	&\les \Big(  \frac{1}{\dl}+\frac{1}{\gamma}   \Big)^{2},
	\end{align*}

\noi
which shows \eqref{pri0}.
Next, we proceed to demonstrate that for $s \geq \frac{3}{4}$, the set $\{ u_{\dl} \}_{\dl\geq2} \subset C([0,T];H^{s}(\T))$ is indeed a Cauchy sequence. 
To do so, we apply the triangle inequality and write the following
	\begin{equation*}
	\begin{split}
	\|u_\dl-u_{\gamma}\|_{C_TH^s_x}
	&\les
	\|u_\dl - P_{\leq K} u_\dl  \|_{C_TH_x^s}
	+
	\|P_{\leq K}   u_\dl   -  P_{\leq K}    u_{\gamma}   \|_{C_TH_x^s}\\
	&\quad
	+\| P_{\leq K}    u_\gamma - u_\gamma\|_{C_TH_x^s}
	\end{split}
	\end{equation*}

	\noi
	Let $\eta>0$,
	then there exists $K_{0}$ such that
	for $K\geq K_0$
	we have
	the following:
	\begin{equation*}
	\|u_\dl  -  P_{\leq K}    u_\dl   \|_{C_TH_x^s}+
	\| P_{\leq K}   u_\gamma  - u_  \gamma\|_{C_TH_x^s}
	<\frac{2\eta}{3}.
	\end{equation*}

	\noi
	Notice that \eqref{pri0}
	implies that for all $\dl,\gamma$ 
	such that $2 \leq \dl \leq  \gamma\leq \infty$, 
	there exists a constant 
	$C= C( \|u_0\|_{H^s})$  independent of $\dl$ and $\gamma$ such that
	for any $K$,
	\begin{align*}
	\|P_{\leq K}(u_\dl)-P_{\leq K}(u_\gamma)\|_{C_TH_x^s}\leq
	2K
	\|u_\dl - u_\gamma\|_{C_T H^{s-1}_x}
	<\frac{CK}{\dl}
	\end{align*}

	\noi
	provided that 
	$w(x,0)=0$.
	Now, we choose $K=\dl^{\frac12}$,
	so that 
	as $\dl \to \infty $
	\begin{equation*}
	\|u_\dl - u_\gamma\|_{C_TH^s_x}
	<\eta.
	\end{equation*}

	\noi
	As $\eta$ is arbitrary, hence we finish the proof for $k\geq 2$ and $s\geq \frac34$.

When $k = 2$ and $s>\frac12$, we follow the came strategy as above.
As we saw above, equation  \eqref{L1} will appear to be the same and linear perturbation can be done in the same way.
Then, we replace the nonlinear perturbation for the $k=2$ case as in Proposition \ref{PRO:EN2}.

\end{proof}

To conclude the proof of Theorem \ref{THM:1}.
 By invoking Proposition \ref{LM:ca1}, there exists a function $u\in C([0,T];H^s(\T))$ such that as $\dl$ approaches infinity, $u_\dl$ converges to $u$ in the space $C([0,T];H^s(\T))$. Our goal now is to demonstrate that $u$ is indeed a solution to the gBO equation.
We observe from \eqref{defq3} that for any value of $\dl \geq 2$, the following inequality holds:
\begin{equation}
\| \cK_\dl u_\dl \|_{C_{T}H_{x}^{s}}\leq \frac{2}{\dl}\|u_\dl\|_{C_{T}H_{x}^s}\leq \frac{c}{\dl},
\label{PE1bb}
\end{equation}

\noi
for some universal constant $c$.
Thus, it becomes evident that $u$ is indeed the solution to the gBO equation with the initial data $u_0$. This is because we have established that the gILW equation can be represented as a perturbed gBO equation.
\begin{equation*}
\partial_t u_\dl +\mathcal H (\partial_x^2 u_\dl )+\dx (u_\dl^k) +
\cK_\dl( \dx u_\dl)
=0 .
\end{equation*}

\noi
Due to the almost everywhere convergence of the linear part, the following convergence is achieved as $\dl\to \infty$:
\begin{align*}
\dt {u_\dl} + \mathcal H \dx^2   {u_{\dl}}  + \dx (u_\dl^k)
\ 
\xrightarrow{  \mathcal D' }  
\  \dt u  + \mathcal H \dx^2 u + \dx u^k,
\end{align*}

\noi
i.e. convergent in the distributional sense. 
Furthermore, as indicated in equation \eqref{PE1bb}, $\cK_\dl(\dx u_\dl)$ vanishes as $\dl\to \infty$.
 As a result, it can be concluded that $u\in C([0,T];H^s(\T))$ is a solution to the gBO equation.

\subsection{The shallow-water limit}
\label{SUB:limit2}

In this subsection, we aim to compare the solutions of the gILW equation to those of the gKdV equation as the limit $\dl \to 0$ is approached. As was discussed in subsection \ref{SUB:LE}, a rescaling of the gILW equation is necessary, which is given by:
\begin{equation}
\label{gILWs}
\dt v_\dl -  \frac{3}{  \dl} \G   (\dx^2 v_\dl) = \dx(v_\dl^k).
\end{equation}

\noi
It is worth mentioning from Lemma \ref{LEM:p1d} that
	\begin{align}
	\label{sys1}
\frac{3}{  \dl}	\ft{  \G \dx }(n)=\frac{3}{  \dl} 
(	n \coth( \dl n)-	\frac{1}{\dl} )=
  n^2
	-  n^2 \frac{h(n,\dl)}{\dl} ,
	\end{align}

\noi
where $h(n,\dl)$ is a bounded function that approaches $O(\dl^3)$ as $\dl \to 0$, uniformly for all values of $n$ in any bounded set of $\R$. This property is discussed in more detail in Remark \ref{RM:p1}.

In the shallow-water limit, one of the key challenges is that uniform convergence (with respect to frequency $n$) in $\dl$ is not guaranteed, as opposed to the deep-water case. To address this, we need to perform a frequency truncation argument, which can be explained as follows:
Our goal is to prove that, given $\eps > 0$, $\| v_\dl- v_\g\|_{C_{T}H^{s}} \les \eps$. This will break into two steps:
(i) Show that, given $\eps > 0$, there exists $N = N(\eps)$ such that $\| v_\dl - v_{\dl, N}\|_{C_{T} H^{s} } \les \eps$, uniformly in $0< \dl < 1$. 
 Here, $v_{\dl, N}$ is a solution to the truncated equation.  Thus, in estimating the difference of the nonlinearities, we have the difference between the low-frequency part and also the high-frequency part of nonlinearity $f(v_\dl)=(u_\dl)^k$, where the latter is to be controlled uniformly in $\dl$ and shown to be less than  $\eps$ for large $N$ (via the Koch-Tzvetkov argument as discussed in Remark \ref{RM:cts}).
(ii) Show that with the frequency truncation parameter $N$ as above, there exists $\dl_0 > 0$ such that $\| v_{\dl, N} - v_{\g, N}\|_{H^{s}(\T)} \les \eps$ for any $0 < \dl, \g < \dl_0$.

For any values of $0<\dl,\gamma\ll 1$, let $v_\dl$ and $v_\gamma$ be two solutions of the scaled ${\rm gILW}_{\dl}$ and scaled $\textrm{gILW}_{\g}$ equations, respectively, in the form of \eqref{gILWs}, with the same initial data. Then, the difference between these two solutions, $w = v_{\dl} - v_{\g}$, satisfies the following equation:
\begin{equation}
\label{ILW6}
\begin{cases}
\dt w + \dx^3 w 
+H_\dl( \dx w)
=(H_\gamma-H_\dl)\dx v_{\g}
- \dx (v_{\dl}^k-v_{\g}^k )  \\
w(x,0) = 0,
\end{cases}
\end{equation}

\noi
where the Fourier multiplier defines $H_\dl$:
\begin{equation}
\label{syH}
\ft H_\dl(n):= -{n^2}  \frac{h(n,\dl)}{\dl}
\end{equation}

\noi
Moreover, for convenience, we denote
\begin{equation}
\label{syL}
L_{\dl, \gamma}=H_\gamma-H_\dl.
\end{equation}

\noi
In order to fully utilize our findings on $h(n,\dl)$, we consider the following frequency truncated scaled gILW equation, where a frequency truncation is applied both to the nonlinearity $f(u)=u^k$ and the initial data:
\begin{equation}
\label{Tcay}
\begin{cases}
\dt   v_{\dl,K}   -  \frac{3}{ \dl}  \G   \dx^2   v_{\dl,K}=   	\dx  (f_{K}( v_{\dl,K} ) ), \\
v_{\dl,K}|_{t = 0} = v_{0,K}.
\end{cases}
\end{equation}

\noi
In the shallow-water limit, this leads to the frequency truncated gKdV equation
\begin{equation}
\label{TKDV}
\begin{cases}
\dt   v_{ K}   +    \dx^3   v_{ K}=   \dx  ( f_{K}( v_{ K} ) ), \\
v_{ K}|_{t = 0} = v_{0,K}.
\end{cases}
\end{equation}

\noi
The corresponding solution $v_{\dl,N},v_N$ are supported on frequency  ${|n| \le K}$.
 Additionally, we use $f_{K}=P_{\leq K} f$ to denote the frequency truncation applied to the nonlinearity. The first step in our analysis is the following proposition.
\begin{proposition}
\label{PRO:S1}	
	
Let $k\geq 2$, $s\geq \frac34$ and $K\in 2^{\ZP} $ to be fixed.
Assume that $v_{\dl,K}$ to be the solution of \eqref{Tcay} with initial data $v_{0,K}$.
Then,  for any $0<T<1$, we have that
\begin{align}
\label{eq:sh1}
 \| v_{\dl,K} -v_K \|_{C([0,T];H^s(\T))} \too 0
 \qquad \text{as} \quad \dl \to 0,
\end{align}

\noi
where $v_K$ is the solution of \eqref{TKDV} with initial data $v_{0,K}$.
Moreover, when $k = 2$ only the statement holds for $s > \frac12$.

\end{proposition}

To prove Proposition \ref{PRO:S1}, we will use the following auxiliary lemmas.

\begin{lemma}
	\label{LEM:ca3b}
	Let $k\geq 2$, 
	and   $0<T<1$.
	Assume that $s>\frac12$ and $   v_{\g}\in L^\infty([0,T];H^{s}(\T))$	 is a solution to scaled  
	 gILW$_{\g}$ \eqref{Tcay} associated with   initial data $v_0 \in H^s(\T)$
	Then, for any fixed $K\in 2^{\ZP} $, there exists $0<\dl_{0}\leq 1$ such that for any $0<\dl, \g<\dl_{0}$ we have
	\begin{align}
	\|   v_{\g,K} \|_{N_{T}^{s,\dl}}
	\les \|  v_{\g,K} \|_{N^{s,\g}_{T}},
	\label{256a}
	\end{align}

	\noi
	where
	the implicit constant only depends on $T,K$, uniform for all $0<\dl,\g<\dl_{0}$.
 	Moreover, the following estimate holds:
	\begin{align}
	\|   v_{\g,K} \|_{N_{T}^{s,\dl}}
	\les \|     v_{\g,K}\|_{L^{\infty}_{T} H^{s}_{x}  } +C(\|     v_{\g,K}\|_{L^{\infty}_{T,x} }   )
	\|   v_{\g,K} \|_{L^{\infty}_{T} H^{s}_{x} } .
	\label{256b}
	\end{align}

\end{lemma}

\begin{proof}
	
	Let us recall the definition of $N^{s,\g}_{T}$-space and symbol of scaled gILW \eqref{gILWs}:
	\begin{align}
	\label{Nspace}
	N^{s,\g}_{T}=L^{\infty}_{T}H^{s}_{x}\cap Y^{s-1,1,\g}_{T}
	\qquad
	\textrm{and}
	\qquad
	\pshallow_{\dl}(n) = n^{3}+ n^{3}\frac{h(n,\dl)}{\dl}
	\end{align}

	\noi
Now, by using the definition of $	\pshallow_{\dl}(n)$ we write the following 
	\begin{align}
	\label{ub4a}
	\begin{aligned}
	\tau -    \pshallow_{\dl}(n) &=\tau -   \pshallow_{\g}(n) - \Big( \pshallow_{\dl} -  \pshallow_{\g}(n) \Big)\\
	&=\tau - \pshallow_{\g}(n)  - \Big(n^{3} \frac{h(n,\dl)}{\dl} - n^{3} \frac{h(n,\g)}{\g} \Big).
	\end{aligned}
	\end{align}

	\noi
The function $h(n,\dl)$ is defined in Lemma \ref{LEM:p1d}.
%
	Moreover, 
	$h(n,\dl)$ has a nice decay in $\dl$, provided 
  $n$ in any bounded set of $ \R$.
	In particular, under the assumption that $n\leq K$, we have 
	\begin{align}
	\frac{h(n,\dl)}{\dl} =O(\dl^{2})
	\qquad
	\text{as } 
	\quad
	\dl\to 0.
	\label{ub4}
	\end{align}

	\noi
	Let  $n\leq K$.
	Then, from \eqref{ub4a}, decay of $h(n,\dl)$, and \eqref{ub4}, we have
	\begin{align}
	\label{uu11}
	\jb{ \tau -  \pshallow_{\dl}(n) } \les 
	\jb{\tau -  \pshallow_{\g}(n) }+   O(\dl^{2}) K^{3}
	\qquad
	\text{as}
	\quad
	\dl\to0.
	\end{align}

	\noi
	Therefore, 	for any $ 0<\g,\dl<1$ and
	 using \eqref{uu11}, we obtain the following control in $Y^{s,b,\dl}$-norm:
	\begin{align}
	\begin{aligned}
	\|   v_{\g,K} \|_{Y^{s-1,1,\dl}_{T}  }
	&\les \|   v_{\g,K } \|_{Y^{s-1,1,\g}_{T}  }
	+  O(\dl^{2}) K^{2} \|   v_{\g,K} \|_{Y^{s,0}_{T}  }\\
	&\les  \|  v_{\g,K} \|_{Y^{s-1,1,\g}_{T}  }
	+ O(\dl^{2})  K^{2}    \|   v_{\g,K} \|_{ L^{\infty}_{T} H^{s}_{x}  },
	\end{aligned}
	\label{ub4b}
	\end{align}
	
	\noi
	as
	$\dl\to0$.
	Here, we notice the fact that
	given  $0<\eps\ll1$, for any fixed $K>0$, there exists $\dl_{0}$ such that for any $ 0< \g,\dl< \dl_{0}$,
	we have
	\begin{align}
	\label{suph}
	\sup_{|n|\leq K} K^{2} \frac{h(n,\dl)}{\dl}
	=O(\dl^{2})  K^{2} <\eps 
	\qquad \text{as} \quad\dl\to0.
	\end{align}

	\noi
	Then, \eqref{ub4b}, \eqref{Nspace}, and \eqref{suph} give
	\begin{align}
	\begin{aligned}
	\|   v_{\g,K} \|_{N^{s,\dl} _{T}} &=
	\|  v_{\g,K} \|_{Y^{s-1,1,\dl}_{T}  } +     \|  v_{\g,K} \|_{ L^{\infty}_{T} H^{s}_{x}  }  \\
	&\les    \|    v_{\g,K}   \|_{Y^{s-1,1,\g}_{T}  } + \|  v_{\g,K}  \|_{ L^{\infty}_{T} H^{s}_{x}  } \les \|   v_{\g,K}  \|_{N^{s,\g}_{T}} 
	\end{aligned}
	\label{ub4b1}
	\end{align}

	\noi
	for any $ 0< \g,\dl< \dl_{0}$. 
	This gives \eqref{256a}.
	Moreover, we follow the same $X^{s,b}$-analysis as in Lemma \ref{lem1},
for $  v_{\g,K}$ satisfies the Duhamel formulation of  scaled gILW \eqref{Tcay}
	and we have the following
	\begin{align}
	\begin{aligned}
	\|  v_{\g,K} \|_{Y^{s-1,1,\g}_{T}  }
	\les\|  v_{\g,K} \|_{ L^{\infty}_{T} H^{s}_{x}  }
	+C(\|   v_{\g,K}  \|_{L^{\infty}_{T,x} })
	\|  v_{\g,K}  \|_{ L^{\infty}_{T} H^{s}_{x}  }
	\end{aligned}
	\label{ub4b2}
	\end{align}

	\noi
	Now, from \eqref{ub4b1} and \eqref{ub4b2},
	we can conclude that for any $0<\g,\dl<\dl_{0}$, 
	the following is true
	\begin{align*}
	\|  v_{\g,K} \|_{ N^{s,\dl}_{T}   } \les  \|  v_{\g,K}  \|_{N^{s,\g}_{T}} \les \|  v_{\g,K}   \|_{ L^{\infty}_{T} H^{s}_{x}  } +C(\|  v_{\g,K}   \|_{L^{\infty}_{T,x} })
	\|  v_{\g,K}   \|_{ L^{\infty}_{T} H^{s}_{x}  }.
	\end{align*}

\end{proof}

We also need the following lemma to deal with the difference between the two solutions.

\begin{lemma}
	\label{LEM:ca3c}
	Let $k\geq 2$ and $0<T<1$.
Assume that $s>\frac 12$ and $ v_{\dl},	v_{\g} \in L^\infty([0,T];H^{s}(\T))$
are  solutions to scaled  gILW \eqref{Tcay} 
associated with  initial data $v_{0,\dl}, v_{0,\g} \in H^s(\T)$, respectively.
	Then,  for any fixed
$K\in 2^{\ZP}$ and any $0<\dl, \g<1$, the following holds 
\begin{align}
	\label{257a}
	\begin{aligned}
	\|  w_K \|_{N_{T}^{s-1,\dl}}
	&\les 
	\|     w_K  \|_{ L^{\infty}_{T} H^{s-1}  } +
	O(\dl^{2}) K^{2} \|    v_{\g,K } \|_{L^\infty_T H_x^{s}}
	\\
	&\quad+
	C(\|   v_{\dl,K}\|_{L^{\infty}_{T,x} }   +\|  v_{\g,K}\|_{L^{\infty}_{T,x} } )
	\|   w_K \|_{ L^{\infty}_{T} H^{s-1}_{x}    }
\end{aligned}
\end{align}
	
\noi
as $\dl\to0$, and where $ w_K =   v_{\g,K} -  v_{\dl,K} $.
\end{lemma}

\begin{proof}
Let us consider the difference between the two frequency truncated equations,
	the scaled $\textrm{gILW}_{\g}$  and scaled $\textrm{gILW}_{\dl}$, as given in equation \eqref{Tcay}.
Setting $w_K = v_{\g,K}- v_{\dl,K}$, we obtain the following difference equation:
	\begin{align}
	\label{eq:dfs2}
	\dt   w_{K} - \frac3\dl \G  \dx^{2}   w_{K} + L_{\dl,\g} (\dx    v_{\g,K}) =
 \dx f_{K}(  v_{\g,K}) - \dx f_{K}(  v_{\dl,K}) .
	\end{align}

 \noi
Additionally, it is worth noting that $w_K$ satisfies the Duhamel formulation of equation \eqref{eq:dfs2}. By following the same proof as in \cite[Lemma 3.1]{MV15} and considering the definition of $N^{s-1,\g}_{T}$ as $L^{\infty}_{T}H^{s-1}_{x}\cap Y^{s-2,1,\g}_{T}$, it suffices to estimate $w_K$ in the $Y^{s-2,1,\dl}$-norm. Hence, by using the Duhamel formulation, we proceed with the following 
	computation
	\begin{align}
	\begin{aligned}
	\| w_{K}   \|_{Y^{s-2,1,\dl}_T}
	& \les   \|    v_{0,\g,N} -   v _{0,\dl,N}   \|_{  H^{s-1}  } 
	+       \|     f_{K}(   v_{\g,K})   -    f_{K}(  v _{\dl,K} )   )  \|_{L^2_T H_x^{s-1}}\\
	& \quad+\|  ( H_\gamma-H_\dl  ) \dx    v_{\g,K} \|_{Y^{s-2,0,\dl}_T} \\
	&  \les   \|    w_{K} \|_{  L^{\infty}_{T}H^{s-1}_{x}  } 
	+
	(\|    v_{\dl,K} \|_{L^\infty_T H_x^{s}}     +   \|     v_{\g,K}  \|_{L^\infty_T H_x^{s}})     
	\|     w_{K}  \|_{L^\infty_T H_x^{s-1}}\\
	& \quad+  \|   L_{\dl,\g}    v_{\g,K} \|_{L^\infty_T H_x^{s-1}}  
	\end{aligned}
	\label{west1}
	\end{align}
	
	\noi
We recall the definitions of $H_{\dl}$ and $L_{\dl,\g}$ from equations \eqref{syH} and \eqref{syL}. From Lemma \ref{LEM:p1d}  we see the definition of $h(n,\dl)$ and  under the assumption that $n\leq K$, we can fully use its decay in  $\dl$ property. Hence, we obtain 
\eqref{ub4}. Hence, we obtain
	\begin{align}
	\label{Lest}
	\|   L_{\dl,\g}      v_{\g,K} \|_{L^\infty_T H_x^{s-1}}
	\les O(\dl^{2}) K  \|      v_{\g,K} \|_{L^\infty_T H_x^{s}}.
	\end{align}
	
	\noi
	as
	$\dl\to0$.
	Then,
	we substitute \eqref{Lest} into \eqref{west1} to have
	\begin{align*}
	\begin{aligned}
	\|  w_{K}   \|_{Y^{s-2,1,\dl}_T}
	&  \les  \|    w_{ K}  \|_{  L^{\infty}_{T}H^{s-1}_{x}  } 
	+ O(\dl^{2}) K\|    v_{\g,K } \|_{L^\infty_T H_x^{s}}\\
	&\quad+C(\|    v_{\dl,K} \|_{L^\infty_T H_x^{s}}    
	+   \|   v_{\g,K}  \|_{L^\infty_T H_x^{s}})     
	\|   w_{K}   \|_{L^\infty_T H_x^{s-1}}    \\
	\end{aligned}
	\end{align*}

	\noi
	as $\dl\to 0$. This finished the proof of \eqref{257a}.
	
\end{proof}

\begin{proof}[Proof of Proposition \ref{PRO:S1}]

Let $k\geq 2$, $s\geq \frac34$, $0<T<1$ and fix $K\in 2^{\ZP}$. We first show that for $v_{\dl}\in C([0,T];H^{s}(\T))$ being a solution of the scaled gILW \eqref{gILWs}, the the one-parameter family of solutions $\{ v_{\dl,K} \}_{\dl>0} $ is a Cauchy sequence in $C([0,T];H^{s-1}(\T))$ as $\dl\to 0$.

For any $0<\g<\dl<1$, let us set $w_K = v_{\g,K} - v_{\dl,K}$. Then, we consider the frequency truncated version of equation \eqref{ILW6}
\begin{align}
\label{Tcay2}
\begin{cases}
\dt  w_{K} + \dx^3   w_{K} 
+H_\dl( \dx   w_{K})
=  L_{\dl,\g}   (\dx   v_{\g,K})
-  
\dx (f_{K}(   v_{\dl,K}) - f_{K}(   v_{\g,K}) ) ,\\
w_{K}(x,0)=0.
\end{cases}
\end{align}

 \noi
We use a similar approach as in equation \eqref{L1} for the equation \eqref{Tcay2}:
	\begin{align}
	\label{L11}
	\begin{aligned}
	\| w_{K}(t)\|^2_{H^{s-1}_{x}}
	&\les
	\sum_{1\leq N\leq K} N^{2s-2}
	\| L_{\dl,\gamma}   (\dx P_{N} v_{\g,K}) \|_{L^{2}_{T,x}} 
	\|  P_{N} w_{K} \|_{L^{2}_{T,x}}\\
	\quad&+
	2 \sum_{k\geq 2}
	\max_{i\in \{0,..,k-1\}} I_{k,i,K}^{t,\dl,\g},
	\end{aligned}
	\end{align}

	\noi
	where $I_{k,i,K}^{t,\dl,\g} $ is now defined by
	\begin{align*}
	I_{k,i,K}^{t,\dl,\g}:=
	 \sum_{  1 \leq N\leq K}  
	N^{2(s-1)} \bigg|\int_0^t\int_\T  v_{\dl,K}^i  v_{\g,K}^{k-1-i} w_{K}  P_N^2  \dx w_{K}  dxdt'\bigg| .
	\end{align*}

	\noi
Regarding the first term on the right-hand-side of \eqref{L11}, we apply Cauchy's inequality to obtain:
	\begin{align*}
	\sum_{1\leq N\leq K} N^{2s-2}
	\| L_{\dl,\gamma}  ( \dx  P_{N} v_{\g,K}) \|_{L^{2}_{T,x}} 
	\|  P_{N} w_{K} \|_{L^{2}_{T,x}}
	\les
	c_{1} \|      L_{\dl,\g}   ( v_{\g,K} ) \|^2_{L^{\infty}_{T} H^s_{x}}+
	c_{2}\|w_{K}\|_{L^{\infty}_{T} H^{s-1}_{x}  }^2,
	\end{align*}

	\noi
where $c_{2}\ll1$ is a constant resulting from the application of Cauchy's inequality. 
Additionally, Proposition \ref{PRO:LWP2} states that for any $s\geq \frac34$, we have:
		\begin{align*}
		\{v_\dl\}_{\dl\geq 1} \subset C([0,T];H^s(\T))
			\end{align*}
	is bounded independent of $\dl$.
	In particular, there exists a universal constant $M$ such that
	\begin{equation}
	\label{lc1}
	\|   v_{\dl,K }\|_{C_{T} H^{s}_{x} } \leq 
	\|  v_{\dl }\|_{C_{T} H^{s}_{x} }\les \|v_0\|_{H^s}< M.
	\end{equation}

\noi
Additionally, with the frequency support condition $|n|<K$ in place, for any $0<\dl,\g<1$, we observe a favorable decay of $h(n,\dl)$ with respect to the depth parameter $\dl$. Specifically, we have:
	\begin{align}
	\label{suph2}
	\sup_{|n|\leq K} K^{2} \frac{h(n,\dl)}{\dl}
	=O(\dl^{2})  K^{2}
	\qquad
	\text{as}
	\quad
	\dl\to0.
	\end{align}

\noi
By using \eqref{lc1} and \eqref{suph2}, for any $\eps>0$, there exists $0<\dl_{0}<1$ such that for any $0<\dl,\g<\dl_{0}$, we have:
	\begin{equation}
	\label{L118}
	c_{1} \|      L_{\dl,\g} ( v_{\g,K})  \|^2_{C_{T} H^s_{x}}\les
	\Big\|   K^{2} \frac{h(\dl,n)}{\dl}      \ft v_{\dl,K}   \Big\|^2_{C_{T} H_{x}^{s}}
	<\eps.
	\end{equation}

	\noi
Regarding the second term on the right-hand side of \eqref{L11}, we proceed exactly as in \eqref{L1}, which can be controlled through \eqref{eq:dfmain}. Hence, by utilizing \eqref{L118} and \eqref{eq:dfmain} in \eqref{L11}, we have, for $c_{2}\ll1$:
	\begin{equation}
	\begin{split}
	\|  w_{K} \|^{2}_{C_T H^{s-1}_{x}}
	&\leq \eps +
	Tc_{2}\|   w_{K} \|^{2}_{C_T H_x^{s-1}}\\
	&\quad+
	T^{\frac14}
	C(\|  v_{\g,K }\|_{N^{s,\dl}_T}+\|   v_{\dl,K } \|_{N^{s,\dl}_T})
	\|  w_{K}  \|_{N^{s-1,\dl}_T}
	\|  w_{K}  \|_{C_T  H_x^{s-1}}.
	\end{split}
	\label{Ldif2}
	\end{equation}

	\noi
	From Lemmas \ref{lem1} and \ref{LEM:ca3b}, there exists $\dl_{0}$ such that
	for any $0< \dl,\g< \dl_{0}<1$,
	\begin{align}
	\begin{aligned}
	\|  v_{\g,K }\|_{N^{s,\dl}_T}+\|     v_{\dl ,K} \|_{N^{s,\dl}_T}
	&\leq
	C_{1} \| v_{\g,K }\|_{C_{T} H^{s}_{x}  } + \|  v_{\dl,K }\|_{N^{s,\dl}_T}
	\leq C,
	\end{aligned}
	\label{dif00a}
	\end{align}
	
	\noi
where the constants are coming from \eqref{256b} and \eqref{256a}.
Next, we need to estimate the difference $\|w_{K}\|_{N^{s-1,\dl}_T}$ in \eqref{Ldif2}, which follows from equation (\ref{257a}) and the condition $w(x,0) = 0$. 	
In particular, by combining equation (\ref{257a}) with equation (\ref{lc1}), we see that there exists a universal constant $\widetilde{C} > 0$ such that
	\begin{align*}
	\|   w_{K}\|_{N_{T}^{s-1,\dl}}
	\les \wt C \|   w_{K} \|_{ C_{T} H^{s-1}_{x}    }+
	O(\dl^{2}) K^{2} \|    v_{\g,K } \|_{C_T H_x^{s}},
	\end{align*}

	\noi
	as $\dl\to 0$.
	Moreover, as in equation (\ref{L118}), given $\eps > 0$, there exists $\delta_0 > 0$ such that for any $0 < \delta, \gamma < \delta_0$, the following holds:
	\begin{align}
	\|  w_{K}\|_{N_{T}^{s-1,\dl}}
	&\les \|   w_{K} \|_{ C_{T} H^{s-1}_{x}    }
	+\eps
	\label{dif00b}
	\end{align}

	\noi
Therefore, by substituting equations (\ref{dif00a}) and (\ref{dif00b}) into equation (\ref{Ldif2}), we obtain:
	\begin{align}
	\begin{aligned}
	\| w_{K} \|^2_{C_T H^{s-1}_x} &\les 
	\eps+
	Tc_{2}\|   w_{K} \|^{2}_{C_T H_x^{s-1}}\\
	&\quad+
	T^{\frac14}
	(    \|   w_{K} \|_{ C_{T} H^{s-1}_{x}    } +\eps )
	\|   w_{K} \|_{C_T H_x^{s-1}},
	\end{aligned}
	\label{fes1}
	\end{align}

\noi
where $c_{2}\ll 1$.
By applying Cauchy's inequality to the last term of the right-hand side of equation (\ref{fes1}), we can conclude that for any $\eps > 0$, there exists $\delta_0 > 0$ such that for any $0 < \delta, \gamma < \delta_0 < 1$, we have that
$	\| w_{K} \|_{C_T H^{s-1}_x} \les 	\eps$. As $\eps>0$ is arbitrary, the one-parameter family $\{ v_{\dl,K} \}_{\dl>0} $ is Cauchy in $C([0,T]; H^{s-1}(\mathbb{T}))$ as $\delta \to 0$. Hence, $ v_{\dl,K} $ converges to some function
$\widetilde{v}_{K} \in C([0,T]; H^{s-1}(\mathbb{T}))$. Additionally, given $\eps > 0$, we have
\begin{align*}
\| v_{\dl, K} - 	\wt v_{K} \|_{C_TH^{s}_x}
&\leq (2K)
\|  v_{\dl,K}-  \wt v_{K}   \|_{ C_{T}  H_{x}^{s-1}}
< {\eps}.
\end{align*}	
	
\noi
From equations (\ref{gILWs}) and (\ref{sys1}), it is clear that $\widetilde{v}_K$ is a solution to the initial value problem \eqref{TKDV}. Therefore, by uniqueness, it follows that $v_K = \widetilde{v}_K$, and this concludes the proof of equation \eqref{eq:sh1}.

When $k=2$ and $s>\frac 12$. We make the similarly changes as we saw in Proposition \ref{LM:ca1} we can conclude the proof.
	
\end{proof}

The final result of our analysis concerns the convergence of solutions to the gILW equation to those of the gKdV equation.

\begin{proposition}
	\label{PRO:maincay}

Let $k\geq 2$ and $s \geq \frac{3}{4}$.
Assume $v_0 \in H^s(\mathbb{T})$ and let $v_{\delta}$ denote the solution of  scaled gILW  \eqref{gILWs} with initial data $v_0$. Then, for any $0 < T < 1$, it follows that $v_{\delta} \to v$ in $C([0,T]; H^s(\mathbb{T}))$ as $\delta \to 0$, where $v$ is the solution of  gKdV  \eqref{KDV} with initial data $v_0$.	
Moreover, when $k=2$ only the statement holds for $s>\frac12$.

\end{proposition}

\begin{proof}

The argument will be the same for the general case when $k\geq 2$ with $s>\frac34$,
 and $k=2$ with $s>\frac12$.	
With loss of generality, we consider the case for $k\geq 2$, $s \geq \frac{3}{4}$.
Let $v_{0,K}$ convergence to  $v_0$ in $H^{s}(\T)$, as $K\to \infty$. Let $v_{\delta,K}$ denote the solution of equation (\ref{Tcay}) with initial data $v_{0,K}$, and let $v_K$ denote the solution of equation (\ref{TKDV}) with initial data $v_{0,K}$. Then, we have:
	\begin{align*}
	v_{\dl} - v&=
	(v_{\dl} -v_{\dl,K})
	+(v_{\dl,K} - P_{\leq N} v_{\dl,K} )\\
	&\quad+P_{\leq N}( v_{\dl,K} -v_{K} )+
	(P_{\leq N} v_{K}  -v_{K} )
	+ ( v_{K} -v)
	\end{align*}

\noi
According to the local well-posedness theory developed in \cite{MT}, for
any $\eps > 0$, there exists a sufficiently large value of $K_{1} = K_{1}(\eps)$ such that for $s \geq \frac34$ and $0 < T < 1$,
 it follows that:
	\begin{align}
	\label{fn1}
	\|v_{K_1} - v \|_{C_{T}  H^{s}_{x} } <\frac{\eps}{5}.
	\end{align}

	\noi 
Furthermore, according to Proposition \ref{PRO:LWP2}, for any $\eps > 0$, there exists a sufficiently large value of $K_{2} = K_{2}(\eps)$ such that for $s \geq \frac34$ and $0 < T < 1$, the following holds:
	\begin{align}
	\label{fn2}
	\|v_{\dl,K_2} - v_{\dl} \|_{C_{T}  H^{s}_{x} } <\frac{\eps}{5}
	\end{align}
	
	\noi 
	uniformly for all $0<\dl<1$.
Moreover, for a fixed $K > 0$ such that $K = \max(K_1,K_2)$ and a given $\eps > 0$, there exists a sufficiently close value of $N = N(\eps)$ to $K$ such that:
	\begin{align}
	\label{fn3}
	\|v_{\dl,K} - P_{\leq N} v_{\dl,K} \|_{C_{T}  H^{s}_{x} } +
	\|P_{\leq N} v_{K}  -v_{K} \|_{C_{T}  H^{s}_{x} }  <\frac{2\eps}{5}
	\end{align}

	\noi
Finally, assuming that we have selected positive constants $N$ and $K$ for a given $\eps > 0$ such that inequalities \eqref{fn1}, \eqref{fn2}, and \eqref{fn3} hold, it follows from Proposition \ref{PRO:S1} that as $\delta \to 0$,	
	\begin{align*}
	\| P_{\leq N} ( v_{\dl,K}  -  v_{K} )  \|_{C_TH^s_x}
	  <\frac{\eps}{5},
	\end{align*}

	\noi
As $\eps > 0$ was arbitrary, we can deduce that
\begin{align*}
\lim_{\delta \to 0}	\|  v_{\delta} - v  \|_{C_TH^s_x} =0,
\end{align*}
thereby completing the proof of the proposition for $k\geq 2$ and $s\geq \frac34$.

The proof for the case where $k=2$ and $s>\frac{1}{2}$ can be derived applying the same argument as in the above discussion.
\end{proof}

To this end, we finished the proof of Theorem \ref{THM:2}.

\begin{proof}[Proof of Corollary \ref{COR:4}]

To prove this corollary, let us consider a fixed BO initial condition $u_{\infty,0}$, and let $ T_\infty$ represent the local existence time of the BO equation. Based on our assumption, we have $u_{\dl,0} \to u_{\infty,0}$ in $H^{s}$.
Then,
there exists some $\dl_{1}\geq 2$ such that for any $\dl\geq \dl_1$, we have $\| u_{\dl,0} \|_{H^{s}} \leq 2 \|  u_{\infty,  0} \|_{H^{s} }$. Consequently, 
the uniform local well-posedness theorem implies the local existence time $T'$ for the ILW equation only depends on $ 2\|u_{\infty,  0}\|_{H^{s}}$, for any $\dl \geq \dl_1$. 
 We can then take $T=T_{\infty} \wedge T' $ to be the common local existence time for any $\dl \geq \dl_1$.

  The method developed in Section \ref{SEC:limit1} then can be used with an additional assumption on the convergence of the ILW initial data. Specifically, we may assume that $\{ u_{\dl,0} \}_{\dl\geq 2} \subset H^{s}(\M)$ forms a Cauchy sequence. In practice, when analysing 
the nonlinear interactions as described in equation \eqref{eq:dfmain},
in order to accommodate various initial conditions,
 it now needs to keep the difference in the initial data term, as presented in Proposition \ref{PRO:EN2}.
 However, this does not affect the validity of the proof, as our assumption of convergence of the initial data ensures that this difference is insignificant. The rest of the proof remains unaffected.

\end{proof}

\begin{ack}\rm 
The author extends his gratitude to his advisors, Tadahiro Oh and Yuzhao Wang, for posing this problem and for their constant support. 
Special thanks are given to Professor Jean-Claude Saut for his assistance in providing references \cite{KS2022, Saut2019} and for clarifying the commonly recognised name for the model equation. 
The author also appreciate the anonymous referees for their carefully reading and comments.  
G.L.~was supported by the Maxwell Institute Graduate School in Analysis and its Applications, 
a Centre for Doctoral Training funded by the UK Engineering and Physical Sciences Research Council (Grant EP/L016508/01), 
the Scottish Funding Council, Heriot-Watt University and the University of Edinburgh;
the EPSRC New Investigator Award (grant no. EP/S033157/1);
the European Research Council (grant no. 864138 ``SingStochDispDyn'').

\end{ack}

\end{document}